\documentclass[review]{elsarticle}

\usepackage{lineno,hyperref}
\modulolinenumbers[5]
\journal{SIAM Journal on Applied Mathematics}

\usepackage{amsmath,amssymb,amsthm,amsfonts,afterpage,float,bm,stmaryrd,paralist,wrapfig,bbm}
\usepackage{algorithm,algpseudocode,arydshln,blkarray,array,tcolorbox}

\usepackage{cancel}
\usepackage{xcolor}
\usepackage{fancyhdr}
\usepackage{graphics}            
\usepackage{hyperref}  
\usepackage{graphicx,epsf}
\usepackage{makeidx}
\usepackage{epsfig}
\usepackage{lscape}
\usepackage{colonequals}
\usepackage{enumerate}
\usepackage{upgreek}

\usepackage{listings} 

\usepackage{algorithm}
\usepackage{algpseudocode}
\usepackage{arydshln}

\allowdisplaybreaks
\pssilent

\newtheorem{theorem}{Theorem}[section]
\newtheorem{lemma}{Lemma}[section]

\newtheorem{remark}{Remark}[section]

\numberwithin{equation}{section}
\numberwithin{figure}{section}
\numberwithin{table}{section}

\def\XXint#1#2#3{{\setbox0=\hbox{$#1{#2#3}{\int}$}
\vcenter{\hbox{$#2#3$}}\kern-.51\wd0}}

\graphicspath{{Figure/}}

\begin{document}

\setlength{\pdfpageheight}{\paperheight}
\setlength{\pdfpagewidth}{\paperwidth}
\title{Supervised Optimal Transport}

\author{Zixuan Cang}
\address{Department of Mathematics, North Carolina State University, Raleigh, North Carolina, 27695}
\author{Qing Nie\fnref{myfootnote1}}
\address{Department of Mathematics, The NSF-Simons Center for Multiscale Cell Fate Research, University of California, Irvine, Irvine, California, 92697}
\author{Yanxiang Zhao\fnref{myfootnote2}}
\address{Department of Mathematics, The George Washington University, Washington D.C., 20052}
\fntext[myfootnote1]{Corresponding author: qnie@uci.edu}
\fntext[myfootnote2]{Co-corresponding author: yxzhao@email.gwu.edu}

\begin{abstract}

Optimal Transport, a theory for optimal allocation of resources, is widely used in various fields such as astrophysics, machine learning, and imaging science. However, many applications impose elementwise constraints on the transport plan which traditional optimal transport cannot enforce. Here we introduce Supervised Optimal Transport (sOT) that formulates a constrained optimal transport problem where couplings between certain elements are prohibited according to specific applications.  sOT is proved to be equivalent to an $l^1$ penalized optimization problem, from which efficient algorithms are designed to solve its entropy regularized formulation. We demonstrate the capability of sOT by comparing it to other variants and extensions of traditional OT in color transfer problem. We also study the barycenter problem in sOT formulation, where we discover and prove a unique reverse and portion selection (control) mechanism. Supervised optimal transport is broadly applicable to applications in which constrained transport plan is involved and the original unit should be preserved by avoiding normalization.

\end{abstract}


\date{\today}
\maketitle

\section{Introduction}

Optimal transport (OT) is a powerful tool for geometrically comparing and connecting measures. It seeks a globally optimal coupling between two probability distributions that minimizes the total coupling cost given a predefined finite cost \cite{Monge_1781,Kantorovitch_1942,Brenier_1991,Villani_2003}.
OT has been successfully applied in many fields recently such as astrophysics \cite{Frisch_2012}, machine learning \cite{Cuturi_2013,arjovsky2017wasserstein,courty2016optimal}, and imaging science \cite{Ferradans_2014,bauer2015diffeomorphic,karlsson2017generalized}. 
The original OT is a linear programming problem which has a computational complexity of $O(n^3)$ \cite{Villani_2003}. Recently, significant advancements in OT computation have been made which enables the application of OT to large scale practical problems, for example, the Sinkhorn algorithm \cite{Cuturi_2013}, Greenkhorn algorithm \cite{altschuler2017near}, and others \cite{cuturi2014fast,pmlr-v108-guo20a,dvurechensky2018computational,lin2019efficient,guo2020fast}.

However, there are limitations of OT that hinder its application to many problems, leading to several variants and extensions of OT. For example, unbalanced optimal transport was introduced to couple non-probability measures and reduce noise in transport plan by replacing the original marginal constraints by soft divergence constraints \cite{chizat2018unbalanced,Chizat_2018}. Partial optimal transport generalizes OT to optimize the transport plan under the condition that a given fraction of mass is transported \cite{Benamou_2015,bonneel2019spot,chapel2020partial}. From the dynamics model perspective, unnormalized OT was introduced to derive the transport dynamics between two marginals of different total mass with an external spatial-dependent or spatiotemporal-dependent mass source \cite{Ganbo_2019,Lee_2021}. In summary, these OT variants relax the marginal mass conservation constraint in the original OT to handle problems where the total masses of the two marginals do not match. 

Another major limitation of OT is that there are natural constraints on the transport plans in many applications which cannot be handled by current OT methods. For instance, when ground transportation is blocked after a major natural disaster, many locations with a demand for resources cannot be safely reached by certain supply distribution locations. In the corresponding optimal transport formulation, there should be constraints on the transport plan, causing some entries in the transport plan being occluded as zero. This leads to a challenging optimal transport problem since the total possible transported mass becomes an unknown due to the elementwise blockages in the transport plan.

Here, we introduce supervised optimal transport (sOT) which supervises the transport plan by enforcing a given elementwise constraint on the transport plan. sOT optimizes both the total transported mass and the transport plan simultaneously.  Different from the OT problems with prescribed inequality constraints \cite{Benamou_2015}, the inequality constraints in sOT, arising due to the infinity entries in the cost matrix, is implicitly determined through the optimization of the transport plan. We show that sOT can be equivalently reformulated and link to the unbalanced OT framework \cite{Chizat_2018}. We further extend the standard OT barycenter problem into sOT barycenter problem, in which an interesting and novel {\it reverse and portion selection mechanism}  is discovered. We propose several new numerical methods for entropy regularized sOT based on Dykstra iteration.

We validate sOT and the proposed numerical algorithms in several numerical experiments. By applying it to an important problem in imaging science, the color transfer problem, we show the benefit and unique capability of sOT over other variants and extensions of traditional OT. More importantly, we prove the reverse and portion selection mechanism for the sOT barycenter problem, which is further validated in detail by numerical examples.

\section{Supervised optimal transport}

In this section, we define the supervised optimal transport (sOT) and derive an equivalent formulation upon which efficient algorithms are derived.

\subsection{Definition of sOT}

Let $\mathbb{R}^n_{+}$ (and $\mathbb{R}^n_{++}$, respectively) denote the $n$-dimensional nonnegative (and positive, respectively) vector space. We define the probability simplex (and strictly positive probability simplex, respectively) as
\begin{align}
\Sigma^n_{+} = \Big\{ \textbf{a} = (a_i)_i\in\mathbb{R}^n_{+}: \sum_i a_i = 1\Big\}, \quad \Sigma^n_{++} = \Big\{\textbf{a} = (a_i)_i \in\mathbb{R}^n_{++}: \sum_i a_i = 1\Big\}.
\end{align}
The polytope of the couplings between $(\textbf{a},\textbf{b})\in \mathbb{R}^n_{+}\times \mathbb{R}^m_{+}$ is defined as
\[
\mathbf{U}(\textbf{a},\textbf{b}) = \big\{ \textbf{P} \in \mathbb{R}_{+}^{n\times m}: \textbf{P}\mathbbm{1} = \textbf{a}, \textbf{P}^T\mathbbm{1} = \textbf{b}  \big\},
\]
where $\textbf{P}^T$ is the transpose of $\textbf{P}$ and $\mathbbm{1}$ is the all-ones matrix. The dimension of $\mathbbm{1}$ is determined by dimension consistency of matrix multiplication in the context. We further define the following polyhedra
\begin{align*}
\mathbf{U}(\le \textbf{a},\le \textbf{b}) = \big\{ \textbf{P} \in \mathbb{R}_{+}^{n\times m}: \textbf{P}\mathbbm{1} \le \textbf{a}, \textbf{P}^T\mathbbm{1} \le \textbf{b}  \big\}, \\
\mathbf{U}(= \textbf{a},\le \textbf{b}) = \big\{ \textbf{P}\in \mathbb{R}_{+}^{n\times m}: \textbf{P}\mathbbm{1} = \textbf{a}, \textbf{P}^T\mathbbm{1} \le \textbf{b}  \big\}.
\end{align*}
We denote by $\iota_{\mathcal{C}}$ the indicator of a set $\mathcal{C}$,
\[
\iota_{\mathcal{C}}(x) =
\begin{cases}
 0, &\text{if}\ x\in \mathcal{C} \\
\infty, & \text{otherwise}.
\end{cases}
\]
For $\textbf{P} = (P_{ij})\in\mathbb{R}_+^{n\times m}$, we define its entropy as
\[
H(\textbf{P}) = -\sum_{i,j} P_{ij}(\log P_{ij} - 1),
\]
in which we use the convention $0\log0 = 0$. The Kullback-Leibler (KL) divergence between $\textbf{P} = (P_{ij})\in\mathbb{R}_+^{n\times m}$ and $\textbf{Q} = (Q_{ij})\in\mathbb{R}_{++}^{n\times m}$ is defined as
\[
\text{KL}(\textbf{P}|\textbf{Q}) = \sum_{i,j} P_{ij}\log\left( \frac{P_{ij}}{Q_{ij}} \right) - P_{ij} + Q_{ij}.
\]
For two vectors $\textbf{u} = (u_i),\textbf{v} = (v_i)$ of the same dimension, we denote entrywise multiplication and division by
\[
\mathbf{u}\odot \mathbf{v} = (u_iv_i)_i, \quad \mathbf{u}./\mathbf{v} = (u_i/v_i)_i.
\] 

For the standard Kantorovich's optimal transport problem with  discrete marginal measures $\textbf{a}, \textbf{b} \in \Sigma_{+}^n$, it reads:
\begin{align}\label{eqn:standardOT}
L_{\text{OT}}(\textbf{a},\textbf{b};\mathbf{C}) = \min_{\textbf{P}\in \textbf{U}(\textbf{a},\textbf{b})} \langle \textbf{P}, \textbf{C} \rangle.
\end{align}
In the framework of sOT, the marginal measures $(\textbf{a},\textbf{b})\in\mathbb{R}_+^n\times\mathbb{R}_+^m$ do not necessarily have the same sum, and we are interested in the cost matrix $\textbf{C} = (C_{ij})$ that contains $\infty$ entries. The $\infty$-pattern of $\textbf{C}$ is defined as the set \cite{Brualdi_2006}
\begin{align}\label{eqn:infinifty_pattern}
\mathcal{P}_{\infty}(\mathbf{C})  = \{ (i,j): C_{ij} = \infty, \quad i = 1,2,\cdots,n, \quad j = 1,2,\cdots,m \},
\end{align}
of positions of $\textbf{C}$ containing an infinity element. Similarly one can define 0-pattern of a transport plan $\textbf{P}$. By virtue of the optimal transport, the 0-pattern of the optimal plan $\textbf{P}^*$ must contain the $\infty$-pattern of $\textbf{C}$, namely, $\mathcal{P}_0(\textbf{P}^*)\supseteq \mathcal{P}_{\infty}(\textbf{C})$. We further define a feasible set $\mathcal{A}_{\textbf{C}}$ for the {\it{marginal blocked distribution}} $(\boldsymbol\upmu,\boldsymbol\upnu)$ as follows:
\begin{align}
\mathcal{A}_{\textbf{C}} = \Big\{ (\boldsymbol\upmu,\boldsymbol\upnu)\in[0,\textbf{a}]\times[0,\textbf{b}] \ \Big|\ \exists \textbf{P} \in \mathbf{U}(\textbf{a}-\boldsymbol\upmu, \textbf{b}-\boldsymbol\upnu) \text{\ such\ that\ } \langle \textbf{P}, \textbf{C} \rangle < \infty \Big\},
\end{align}
in which the inclusion $(\boldsymbol\upmu,\boldsymbol\upnu)\in[0,\textbf{a}]\times[0,\textbf{b}]$ is entrywise, namely, $\mu_i\in[0,a_i], i=1,2,\cdots n$ and $\nu_j\in[0,b_j], j=1,2,\cdots m$. 

We define sOT as the following minimization problem
\begin{align}\label{eqn:UUOT_formI}
L_{\text{sOT}}(\textbf{a},\textbf{b};\mathbf{C}) := \min_{(\boldsymbol\upmu,\boldsymbol\upnu)\in\mathcal{B}} \min_{\mathbf{P}\in \mathbf{U}(\textbf{a}-\boldsymbol\upmu, \textbf{b}-\boldsymbol\upnu)} \langle \textbf{P}, \textbf{C} \rangle 
\end{align}
where $\mathcal{B}:= \text{argmin}_{(\boldsymbol\upmu,\boldsymbol\upnu)\in\mathcal{A}_{\textbf{C}}}\  \|\boldsymbol\upmu\|_1 + \|\boldsymbol\upnu\|_1$. In other words, we aim to find the optimal transport plan $\textbf{P}$ which transports the most marginal density (blocks the least $(\boldsymbol\upmu,\boldsymbol\upnu)$) with minimal cost.

\begin{remark}
Note that when the cost matrix $\emph{\textbf{C}}$ contains $\infty$, $(\boldsymbol\upmu,\boldsymbol\upnu) = (\mathbf{0},\mathbf{0})$ may still be a feasible point in $\mathcal{A}_{\emph{\textbf{C}}}$ (and therefore in $\mathcal{B}$). For example, one can consider
\begin{align*}
\emph{\textbf{C}} =      \left[
   \begin{array}{ccc}
    1 & \infty & 1 \\
    1 & 1 & 1 \\     
    1 & 1 & 1 
   \end{array}
   \right],\ 
\emph{\textbf{a}} =      \left[
   \begin{array}{ccc}
    0.2 \\
    0.3 \\     
    0.5
   \end{array}
   \right],\   
\emph{\textbf{b}} =      \left[
   \begin{array}{ccc}
    0.4 \\
    0.3 \\     
    0.3
   \end{array}
   \right],\      
\emph{\textbf{P}} =      \left[
   \begin{array}{ccc}
    0.1 & 0 & 0.1 \\
    0.1 & 0.1 & 0.1 \\     
    0.2 & 0.2 & 0.1 
   \end{array}
   \right],
\end{align*}
then $(\boldsymbol\upmu,\boldsymbol\upnu)=(\mathbf{0},\mathbf{0})\in\mathcal{A}_{\emph{\textbf{C}}}$ since $\emph{\textbf{P}}\in \mathbf{U}(\emph{\textbf{a}},\emph{\textbf{b}})$ and $\langle \emph{\textbf{P}}, \emph{\textbf{C}}\rangle$ is finite. In this case, sOT reduces to the standard OT as $\mathbf{U}(\emph{\textbf{a}}, \emph{\textbf{b}})$ is non-empty and the minimum is reached at some optimal $\emph{\textbf{P}}^*$. 
\end{remark}

In practical applications, the cost matrix $\textbf{C}$ could be sparse with respect to the $\infty$ entries, which most likely leads to $(\boldsymbol\upmu,\boldsymbol\upnu)=(\mathbf{0},\mathbf{0})\notin\mathcal{A}_{\textbf{C}}$. On the other hand, $\mathcal{A}_{\textbf{C}}$ is non-empty as $(\boldsymbol\upmu,\boldsymbol\upnu) = (\textbf{a},\textbf{b})$ is always an element in $\mathcal{A}_{\textbf{C}}$. Therefore, we expect to find some $(\boldsymbol\upmu,\boldsymbol\upnu)\in [0,\textbf{a}]\times[0,\textbf{b}]$ with smallest $l^1$-norm, and the associated optimal transport plan $\textbf{P}^*$ over $\mathbf{U}(\textbf{a}-\boldsymbol\upmu,\textbf{b}-\boldsymbol\upnu)$.

A problem related to sOT (\ref{eqn:UUOT_formI}) is the partial transport problem \cite{Benamou_2015} which finds the optimal plan to transport a given fraction of mass instead of the total amount of marginal mass. It is formulated as follows. Given marginal densities $(\textbf{a},\textbf{b})\in\mathbb{R}_{++}^{n}\times\mathbb{R}_{++}^{m}$, not necessarily with the same total mass, the partial transport problem minimizes
\begin{align}\label{eqn:PT}
\min_{\textbf{P}\in \mathbf{U}(\le \textbf{a}, \le \textbf{b})} \{ \langle \textbf{P}, \textbf{C} \rangle: \langle \textbf{P}, \mathbbm{1} \rangle = \theta \}.
\end{align}
in which $\theta$ is the given fraction of mass. Different from partial optimal transport, sOT does not require a given fraction of mass to be transported and instead optimizes the transported mass. Specifically, sOT can be rewritten as
\begin{align}\label{eqn:UUOT_formPT}
\max_{\theta \in [0, \min(\textbf{a}^T\mathbbm{1},\textbf{b}^T\mathbbm{1})]}\min_{\textbf{P}\in \mathbf{U}(\le \textbf{a}, \le \textbf{b})} \{ \langle \textbf{P}, \textbf{C} \rangle: \langle \textbf{P}, \mathbbm{1} \rangle = \theta \}.
\end{align}
where sOT performs an extra maximization over the transported mass $\theta$. Note that if the cost matrix $\textbf{C}$ does not contain $\infty$ entries and $\textbf{a}^T\mathbbm{1}=\textbf{b}^T\mathbbm{1}$ , the reformulated problem (\ref{eqn:UUOT_formPT}) degenerates to the standard OT problem (\ref{eqn:standardOT}) since $\theta$ can be equal to the total mass, the largest possible value. On the other hand, if the cost matrix $\textbf{C}$ contains $\infty$ entries, the partial transport problem (\ref{eqn:PT}) may not be well defined for a given fraction of mass $\theta$ since the feasible set is potentially empty. In this case, sOT (\ref{eqn:UUOT_formPT}) remains wellposed as $\theta=0$ is always a feasible transported mass.

\subsection{Equivalent sOT formulations }

The sOT in the forms (\ref{eqn:UUOT_formI}) or (\ref{eqn:UUOT_formPT}) is a double optimization problem which is computationally challenging. Here, we recast it to the form of a single optimization:
\begin{align}\label{eqn:UUOT_formII}
L_{\text{sOT}}(\textbf{a},\textbf{b};\mathbf{C}) = \min_{\substack{(\boldsymbol\upmu,\boldsymbol\upnu)\in\mathcal{A}_{\textbf{C}}\\ \textbf{P}\in \mathbf{U}(\textbf{a}-\boldsymbol\upmu,\textbf{b}-\boldsymbol\upnu)}} \langle \textbf{P}, \textbf{C} \rangle + \gamma(\|\boldsymbol\upmu\|_1 + \|\boldsymbol\upnu\|_1) 
\end{align}
for sufficiently large $\gamma$ depending on $\mathbf{C}$. The equivalence between (\ref{eqn:UUOT_formI}) and (\ref{eqn:UUOT_formII}) is summarized in lemma \ref{lemma:equivalence}.

\begin{figure}[!htbp]
\centerline{
\includegraphics[width=50mm]{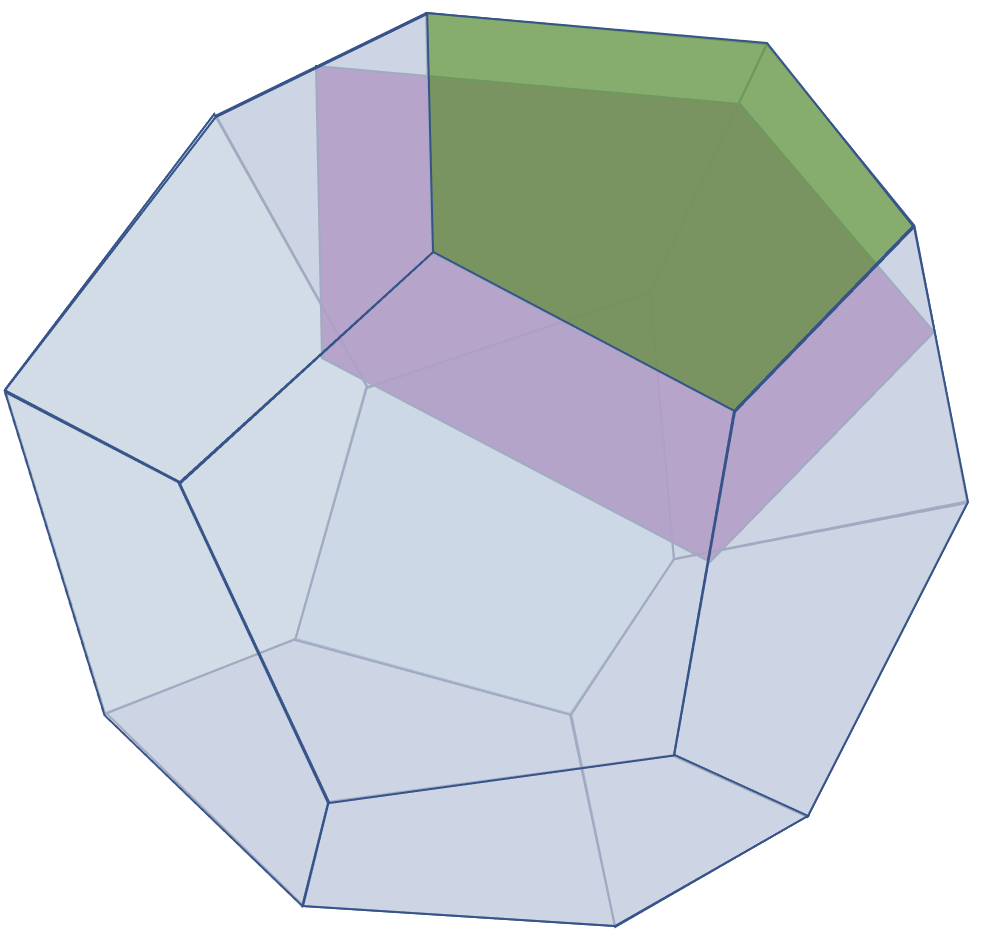}
 }
\caption{A schematic polyhedron for lemma \ref{lemma:polytope}. The green face is the hyperplane $\mathbbm{1}^Tx = \bar{\theta}$. The pink plane is the hyperplane $\mathbbm{1}^Tx = \theta$. When the pink plane is close enough to the green one, the set of extreme points $E_{J_{\theta}}$ is invariant. See the proof of lemma \ref{lemma:polytope} for the definitions of the notations.}
\label{fig:Schematic_Polytope}
\end{figure}

To begin with the proof, we need a lemma characterizing the difference between $\mathbf{P}_{\theta}$ and $\mathbf{P}_{\bar{\theta}}$ for the maximal possible transported mass $\bar{\theta}$ and a mass $\theta$ smaller than but close to $\bar{\theta}$.


\begin{lemma}\label{lemma:polytope}
Let $C=\{x: Ax \le b\}$ be a bounded convex polyhedron where $A = \left[a_1^T | \cdots | a_r^T \right]^T\in\mathbb{R}^{r\times n}$ and $b\in\mathbb{R}^r$. Define
\[
\bar{\theta}: = \max_{\bar{x}\in C}\ \mathbbm{1}^T\bar{x}, \quad C_{\bar{\theta}} = \underset{\bar{x}\in C}{\mathrm{argmax}}\ \mathbbm{1}^T\bar{x},
\]
and 
\[
C_{\theta} = \{x\in C : \mathbbm{1}^Tx = \theta \},
\]
for any possible value of $\theta$ that $\mathbbm{1}^Tx$ can take. Then there exists a critical $\theta_0<\bar{\theta}$, such that for any fixed $\theta\in(\theta_0, \bar{\theta} ]$, and any $x\in C_{\theta}$, there exists a $\bar{x}\in C_{\bar{\theta}}$ such that $\| x - \bar{x} \|_{1}\le \eta |\bar{\theta}-\theta|$, where $\eta$ is independent of $\theta$.
\end{lemma}

\begin{proof}
By definition, $C_{\theta}$ is a bounded convex polyhedron, which can be reformulated as the convex hull of its extreme points:
\[
C_{\theta} = \text{conv}\left(x_{\theta}^{(1)}, x_{\theta}^{(2)}, \cdots, x_{\theta}^{(s(\theta))}\right)
\]
in which $x_{\theta}^{(1)}, x_{\theta}^{(2)}, \cdots, x_{\theta}^{(s(\theta))}$ are all extreme points of $C_{\theta}$. We denote by $J_{\theta}: = \{j: a_j^T x_{\theta}^{(k)} = b_j, \text{ for some } 1\le k \le s(\theta) \}$ the indices of constraints which are saturated in at least one extreme point of $C_{\theta}$. In other words, $J_{\theta}$ are the indices of constraints which interacts with hyperplane $\mathbbm{1}^Tx = \theta$ on $C$. Let $E_C = \{x_1,\cdots, x_m\}$ be the set of extreme points of $C$ such that $C = \text{conv}(E_C)$, we denote by $E_j: = \{x\in E_C: a^T_j x = b_j\}$ the extreme points of $C$ saturating the $j$-th constraint, and $E_{J_{\theta}} = \cup_{j\in J_{\theta}} E_j$.

Evaluating the linear function $\mathbbm{1}^Tx$ over $E_C$, we know that the maximal value over $\{ \mathbbm{1}^Tx_1, \cdots, \mathbbm{1}^Tx_m \}$ equals $\bar{\theta}$. Let the second maximal value over $\{ \mathbbm{1}^Tx_1, \cdots, \mathbbm{1}^Tx_m \}$ be $\theta_0$. Note that in an extreme case where the second maximal value $\theta_0$ cannot be attained, the maximal value $\bar{\theta}$ is reached in the entire $C$, and the conclusion of the lemma is trivially held. So we only consider the nontrival case where the second maximal value $\theta_0$ is attained. It is evident that $E_{J_{\theta}}$ is invariant for any $\theta\in(\theta_0,\bar{\theta}]$.

For any $\theta\in(\theta_0,\bar{\theta})$, we consider the extreme points $\{x_{\theta}^{(j)}\}_{j=1}^{s(\theta)}$ of $C_{\theta}$ in which $s(\theta) \equiv s$. As $\theta\rightarrow\bar{\theta}$, we have $x_{\theta}^{(j)} \rightarrow \bar{x}_{\bar{\theta}}^{(j)}$ for $j=1,\cdots, s$. Here $\{\bar{x}_{\bar{\theta}}^{(j)}\}_{j=1}^{s}$ are the extreme points of $C_{\bar{\theta}}$, some of which might be repeatedly counted. Now for any point $x\in C_{\theta}$ such that
\[
x = \lambda_1x_{\theta}^{(1)} + \cdots + \lambda_s x_{\theta}^{(s)},\quad \{\lambda_j\}_{j=1}^s \in \Sigma_+^s
\]
we can take $\bar{x}$ as
\[
\bar{x} = \lambda_1\bar{x}_{\bar{\theta}}^{(1)} + \cdots + \lambda_s \bar{x}_{\bar{\theta}}^{(s)},
\]
then
\begin{align}\label{eqn:temp01}
\|x-\bar{x}\|_{1}\le \max_{1\le j \le s} \| x_{\theta}^{(j)} - \bar{x}_{\bar{\theta}}^{(j)} \|_{1} \le \|\mathbbm{1}\|_2\cdot \max_{1\le j \le s} \| x_{\theta}^{(j)} - \bar{x}_{\bar{\theta}}^{(j)} \|_{2}.
\end{align}

Now we consider the angle $\sigma_j$ made by $\bar{x}_{\bar{\theta}}^{(j)} - x_{\theta}^{(j)} $ and $\mathbbm{1}$,
\begin{align}\label{eqn:temp02}
\sigma_j = \arccos \frac{\left\langle \bar{x}_{\bar{\theta}}^{(j)} - x_{\theta}^{(j)}, \mathbbm{1}\right\rangle}{\|\bar{x}_{\bar{\theta}}^{(j)} - x_{\theta}^{(j)}\|_2\|\mathbbm{1}\|_2}.
\end{align}
It turns out that $\sigma_j$ cannot be $\pi/2$. Otherwise, 
\[
0 = \left\langle \bar{x}_{\bar{\theta}}^{(j)} - x_{\theta}^{(j)}, \mathbbm{1}\right\rangle = \bar{\theta} - \theta,
\]
implies that $\theta = \bar{\theta}$, which is the trivial case that $\bar{\theta}$ is attained everywhere in $C$. Hence $\{\cos \sigma_j\}_{j=1}^s$ must be bounded away from 0. Then we have
\[
\cos\sigma_j  \|\bar{x}_{\bar{\theta}}^{(j)} - x_{\theta}^{(j)}\|_2 \|\mathbbm{1}\|_2 = \left\langle \bar{x}_{\bar{\theta}}^{(j)} - x_{\theta}^{(j)}, \mathbbm{1}\right\rangle = \bar{\theta} - \theta,
\]
which leads to
\[
\|\bar{x}_{\bar{\theta}}^{(j)} - x_{\theta}^{(j)}\|_2 = \frac{1}{\cos\sigma_j\|\mathbbm{1}\|_2} (\bar{\theta}-\theta).
\]
Combining with equation (\ref{eqn:temp01}), we have
\[
\|x-\bar{x}\|_{1} \le \frac{1}{\min_{1\le j \le s}\cos\sigma_j }|\bar{\theta}-\theta|.
\]
The conclusion holds by taking $\eta = (\min_{1\le j \le s}\cos\sigma_j )^{-1}$. 

Lastly, we show that $\cos\sigma_j$ is independent of $\theta$, or more specifically of $x_{\theta}^{(j)}$. To this end, we choose an arbitrary $\tilde{\theta}\in (\theta,\bar{\theta})$. Note that $x_{\theta}^{(j)}$ and $\bar{x}_{\bar{\theta}}^{(j)}$ saturate the same set of constraints since when $\theta\rightarrow\bar{\theta}$, the hyperplane $\mathbbm{1}^Tx=\theta$ does not go through any extreme point of $C$ ($E_{J_{\theta}}$ is invariant for any $\theta\in(\theta_0,\bar{\theta}]$), therefore the point 
\begin{align}\label{eqn:temp03}
\tilde{x}_{\tilde{\theta}}^{(j)} : = \frac{\tilde{\theta}-\theta}{\bar{\theta}-\theta} \bar{x}_{\bar{\theta}}^{(j)} + \frac{\bar{\theta}-\tilde{\theta}}{\bar{\theta}-\theta} x_{\theta}^{(j)}
\end{align}
which saturate the same set of constraints as $x_{\theta}^{(j)}$ and $\bar{x}_{\bar{\theta}}^{(j)}$, is exactly the extreme point of $C_{\tilde{\theta}}$ that falls on the line segment between $x_{\theta}^{(j)}$ and $\bar{x}_{\bar{\theta}}^{(j)}$. Finally the linear relation (\ref{eqn:temp03}) together with the definition of $\sigma_j$ indicates that $\cos\sigma_j$ is independent of $\theta$. The proof is completed.

\end{proof}

\begin{lemma}\label{lemma:equivalence}
Given a cost matrix $\mathbf{C}$ with $\infty$-pattern $\mathcal{P}_{\infty}(\mathbf{C})$, The two sOT formulations (\ref{eqn:UUOT_formI}) and (\ref{eqn:UUOT_formII}) are equivalent for sufficiently large $\gamma$.
\end{lemma}
\begin{proof}

Step I. Let $(\boldsymbol\upmu_{\text{opt}}, \boldsymbol\upnu_{\text{opt}})$ be an optimal pair of blocked measures for  (\ref{eqn:UUOT_formI}), namely, the system can at most transfer the amount of mass $\bar{\theta} := \|\textbf{a}-\boldsymbol\upmu_{\text{opt}}\|_1 = \|\textbf{b}-\boldsymbol\upnu_{\text{opt}}\|_1$. We denote by $\textbf{P}_{\bar{\theta}}^*$ a corresponding optimal transport plan. Take any nonnegative and feasible $ \theta < \bar{\theta}$, if we can show that for a plan $\textbf{P}_{\theta}^*$ defined as
\begin{align}\label{eqn:mOpt}
\textbf{P}_{\theta}^* = \underset{ \textbf{P} } {\text{argmin}} \{ \langle \textbf{P}, \textbf{C} \rangle: \textbf{P} \in\mathbf{U}(\le \textbf{a}, \le \textbf{b}), \langle \textbf{P}, \mathbbm{1}\rangle = \theta \},
\end{align}
one has that
\begin{align}\label{eqn:gammaInequality}
\langle \textbf{P}^*_{\bar{\theta}} - \textbf{P}^*_{\theta}, \textbf{C} \rangle \le 2\gamma \langle \textbf{P}^*_{\bar{\theta}} - \textbf{P}^*_{\theta}, \mathbbm{1} \rangle,
\end{align}
for some $\gamma>0$, then for any $(\boldsymbol\upmu,\boldsymbol\upnu)$ such that $\theta = \|\textbf{a}-\boldsymbol\upmu\|_1 = \|\textbf{b}-\boldsymbol\upnu\|_1$, it results in
\begin{align*}
&\|\boldsymbol\upmu\|_1 + \|\boldsymbol\upnu\|_1 - ( \|\boldsymbol\upmu_{\text{opt}}\|_1 + \|\boldsymbol\upnu_{\text{opt}}\|_1) \\
=&\  \langle \boldsymbol\upmu - \boldsymbol\upmu_{\text{opt}}, \mathbbm{1} \rangle + \langle \boldsymbol\upnu - \boldsymbol\upnu_{\text{opt}}, \mathbbm{1} \rangle \\
=&\  2 \langle \textbf{P}_{\bar{\theta}}^* - \textbf{P}_{\theta}^*, \mathbbm{1} \rangle \\
\ge &\  \gamma^{-1} \left\langle \textbf{P}_{\bar{\theta}}^* - \textbf{P}_{\theta}^*, \mathbf{C} \right\rangle,
\end{align*}
leading to
\[
\langle \textbf{P}_{\bar{\theta}}^*, \textbf{C} \rangle + \gamma (\|\boldsymbol\upmu_{\text{opt}}\|_1 + \|\boldsymbol\upnu_{\text{opt}}\|_1) \le 
\langle \textbf{P}_{\theta}^*, \textbf{C} \rangle + \gamma (\|\boldsymbol\upmu\|_1 + \|\boldsymbol\upnu\|_1),
\]
which implies the optimality of $(\boldsymbol\upmu_{\text{opt}}, \boldsymbol\upnu_{\text{opt}},\mathbf{P}_{\bar{\theta}}^*)$ for (\ref{eqn:UUOT_formII}),
and consequently the equivalence holds.

Step II. We now prove that there exists a constant $\gamma>0$ such that (\ref{eqn:gammaInequality}) holds. Using lemma \ref{lemma:polytope}, we know that there exists a critical $\theta_0$, such that for any given $\theta\in(\theta_0,\bar{\theta}]$ and $\mathbf{P}_{\theta}^*$ defined in (\ref{eqn:mOpt}), we can find a feasible $\mathbf{P}_{\bar{\theta}}$ satisfying $\langle \mathbf{P}_{\bar{\theta}}, \mathbbm{1} \rangle = \bar{\theta}$, such that
\begin{align}
\| \mathbf{P}_{\bar{\theta}} - \mathbf{P}_{\theta}^* \|_1 \le \eta|\bar{\theta}-\theta|.
\end{align}
Then 
\begin{align}
\langle \textbf{P}^*_{\bar{\theta}} - \textbf{P}^*_{\theta}, \textbf{C} \rangle & \le
\langle \textbf{P}_{\bar{\theta}} - \textbf{P}^*_{\theta}, \textbf{C} \rangle \nonumber\\
& \le \| \mathbf{P}_{\bar{\theta}} - \mathbf{P}_{\theta}^* \|_1 \cdot \|\mathbf{C}\|_{\infty} \nonumber \\
& \le \|\mathbf{C}\|_{\infty}\eta|\bar{\theta}-\theta| \nonumber \\
& = \|\mathbf{C}\|_{\infty}\eta \langle \textbf{P}^*_{\bar{\theta}} - \textbf{P}^*_{\theta}, \mathbbm{1} \rangle. \label{eqn:temp06}
\end{align}
On the other hand, for any feasible $\theta\le\theta_0$ and $\mathbf{P}_{\theta}^*$ defined in (\ref{eqn:mOpt}), we simply have
\begin{align}
\langle \textbf{P}^*_{\bar{\theta}} - \textbf{P}^*_{\theta}, \textbf{C} \rangle & = \langle \textbf{P}^*_{\bar{\theta}}, \textbf{C} \rangle -
\langle \textbf{P}^*_{\theta}, \textbf{C} \rangle  \nonumber\\
& \le \|\mathbf{C}\|_{\infty}\cdot \bar{\theta} \nonumber\\
&=\|\mathbf{C}\|_{\infty}\cdot \frac{\bar{\theta}}{\bar{\theta}-\theta}(\bar{\theta}-\theta) \nonumber\\
&\le \|\mathbf{C}\|_{\infty}\cdot \frac{\bar{\theta}}{\bar{\theta}-\theta_0}(\bar{\theta}-\theta) \nonumber \\
&= \|\mathbf{C}\|_{\infty}\cdot \frac{\bar{\theta}}{\bar{\theta}-\theta_0}\langle \textbf{P}^*_{\bar{\theta}} - \textbf{P}^*_{\theta}, \mathbbm{1} \rangle. \label{eqn:temp07}
\end{align}
Finally, combining the inequalities (\ref{eqn:temp06}) and (\ref{eqn:temp07}) and taking $2\gamma = \max\{\eta, \bar{\theta}/(\bar{\theta}-\theta_0)\}\cdot \|\mathbf{C}\|_{\infty}$, we prove the inequality (\ref{eqn:gammaInequality}), and therefore the optimality of $(\boldsymbol\upmu_{\text{opt}}, \boldsymbol\upnu_{\text{opt}},\mathbf{P}_{\bar{\theta}}^*)$ for (\ref{eqn:UUOT_formII}).

\end{proof}

Note that $\mathcal{A}_{\textbf{C}}$ is always non-empty, we can therefore rewrite the formulation (\ref{eqn:UUOT_formII}) in a simpler form:
\begin{align}\label{eqn:UUOT_formIII}
L_{\text{sOT}}(\textbf{a},\textbf{b};\mathbf{C}) = \min_{\textbf{P}\in \mathbf{U}(\le \textbf{a},\le \textbf{b})} \langle \textbf{P}, \textbf{C} \rangle + \gamma(\|\textbf{a}-\textbf{P}\mathbbm{1}\|_1 + \|\textbf{b}-\textbf{P}^T\mathbbm{1}\|_1).
\end{align}

\section{Entropic regularization of sOT}

The idea to regularize the standard OT problem by an entropic term can be traced back to the early work by Schrodinger \cite{Schrodinger_PhysMath1931}. This entropic regularization has been well motivated in economics for predicting flows of commodities or actors in a market, in which the smoothness of such flows can be guaranteed \cite{Wilson_JTEP1969}. A recent work \cite{Cuturi_2013} provides a new motivation from the computational perspective that entropic regularization defines a strongly convex programming. Unlike the standard OT problem (\ref{eqn:standardOT}) which has multiple solutions, the entropic regularized OT problem has a unique solution, which corresponds to the optimizer of (\ref{eqn:standardOT}) with maximal entropy in the limit as the regularization parameter $\epsilon$ varnishes. More importantly, the unique solution to the entropic regularized OT problem is simply a diagonal scaling of the matrix $e^{-\textbf{C}/\epsilon}$. This diagonal scaling process can be efficiently implemented by the Sinkhorn algorithm \cite{Sinkhorn_AMS1964,Sinkhorn_PJM1967, SInkhorn_AMM1967}, which has linear rate of convergence \cite{Franklin_LAA1989}.

We now consider the entropic regularization of sOT (\ref{eqn:UUOT_formII}):
\begin{align}\label{eqn:EntropicUUOT_formI}
\min_{\substack{(\boldsymbol\upmu,\boldsymbol\upnu)\in\mathcal{A}_{\textbf{C}}\\ \textbf{P}\in \mathbf{U}(\textbf{a}-\boldsymbol\upmu,\textbf{b}-\boldsymbol\upnu)}} \langle \textbf{P}, \textbf{C} \rangle - \epsilon  H(\textbf{P}) + \gamma(\|\boldsymbol\upmu\|_1 + \|\boldsymbol\upnu\|_1) 
\end{align}
or equivalently 
\begin{align}\label{eqn:EntropicUUOT_formII}
\min_{\textbf{P} \in \mathbf{U}(\le \textbf{a}, \le \textbf{b})} \langle \textbf{P}, \textbf{C} \rangle - \epsilon H(\textbf{P}) + \gamma(\|\textbf{a} - \textbf{P}\mathbbm{1} \|_1 + \|\textbf{b} - \textbf{P}^T\mathbbm{1} \|_1) .
\end{align}
It is well known that the unique solution $\textbf{P}_{\epsilon}^*$ of (\ref{eqn:EntropicUUOT_formII}) converges to the optimal solution with maximal entropy within the set of all optimal solutions of the problem (\ref{eqn:UUOT_formIII}) \cite{Cominetti_1994}. 

Taking $\textbf{K} = \exp(-\textbf{C}/\epsilon)$ as the Gibbs kernel, sOT problem (\ref{eqn:EntropicUUOT_formII}) can be rewritten  in terms of the KL divergence as:
\begin{align}\label{eqn:KLUUOT_formI}
\min_{\textbf{P} \in \mathbf{U}(\le \textbf{a}, \le \textbf{b})}  \epsilon \text{KL}(\textbf{P}|\textbf{K}) + \gamma(\|\textbf{a} - \textbf{P}\mathbbm{1} \|_1 + \|\textbf{b} - \textbf{P}^T\mathbbm{1} \|_1),
\end{align}
or equivalently
\begin{align}\label{eqn:KLUUOT_formII}
\min_{\textbf{P} \in \mathbb{R}_{+}^{n\times m}}  \epsilon \text{KL}(\textbf{P}|\textbf{K}) + \gamma\|\textbf{a} - \textbf{P}\mathbbm{1} \|_1 + \iota_{[0,\textbf{a}]}(\textbf{P}\mathbbm{1}) + \gamma \|\textbf{b} - \textbf{P}^T\mathbbm{1} \|_1 + \iota_{[0,\textbf{b}]}(\textbf{P}^T\mathbbm{1}).
\end{align}

\subsection{Dykstra Algorithm}

The entropic regularized sOT (\ref{eqn:KLUUOT_formII}) fits into a more general form
\begin{align}\label{eqn:BregmanProblem}
\min_{\textbf{P} \in \mathbb{R}_{+}^{n\times m}}   B_g(\textbf{P}|\textbf{K}) + \hat{h}_1(\textbf{P}) + \hat{h}_2(\textbf{P}).
\end{align}
Here $g$ is a given proper closed and strictly convex and differentiable function.  $B_g$ is the Bregman divergence (Bregman distance) defined as
\begin{align}\label{eqn:BregmanDis}
B_{g}(\textbf{P}|\textbf{Q}) = g(\textbf{P}) - g(\textbf{Q}) - \langle \nabla g(\textbf{Q}), \textbf{P} - \textbf{Q} \rangle.
\end{align}
Besides, $\hat{h}_1$ and $\hat{h}_2$ are two proper and lower semicontinuous convex functions.

Note that the Legendre transform of $g$
\[
g^*(y) = \max_{x} \langle x, y \rangle - g(x)
\]
is also smooth and strictly convex. In particular one has that $\nabla g$ and $\nabla g^*$ are bijective function such that $\nabla g^* = (\nabla g)^{-1}$.

Define the Bregman proximal operator of a convex function $\phi$ as
\begin{align}
\text{prox}_{\phi}^{B_g}(\textbf{Q}) = \underset{\textbf{P}}{\mathrm{argmin}}\ B_g(\textbf{P}|\textbf{Q}) + \phi(\textbf{P}).
\end{align} 
We assume that $\phi$ is coercive so that $\text{prox}_{\phi}^{B_g}(\textbf{Q})$ is uniquely defined by strict convexity.

The Dykstra algorithm for problem (\ref{eqn:BregmanProblem})  \cite{Peyre_SIIS2015} reads as follows

\begin{tcolorbox}

Dykstra algorithm for (\ref{eqn:BregmanProblem})   \\

Input: $\textbf{P}^0 = \textbf{K}$ and $\lambda^{-1} = \lambda^0 = 0$; \\
General step: for any $k=0,1,2,\cdots$ execute the following steps:
\begin{align}
\textbf{P}^{2k+1} & = \text{prox}_{\hat{h}_1}^{B_g}\Big( \nabla g^*\Big[ \nabla g(\textbf{P}^{2k}) + \lambda^{2k-1} \Big]  \Big); \\
 \lambda^{2k+1} & = \lambda^{2k-1} + \nabla g(\textbf{P}^{2k}) - \nabla g(\textbf{P}^{2k+1}); \\
\textbf{P}^{2k+2} & = \text{prox}_{\hat{h}_2}^{B_g}\Big( \nabla g^*\Big[ \nabla g(\textbf{P}^{2k+1}) + \lambda^{2k} \Big]  \Big); \\
\lambda^{2k+2} & = \lambda^{2k} + \nabla g(\textbf{P}^{2k+1}) - \nabla g(\textbf{P}^{2k+2}).
\end{align}

\end{tcolorbox}

It is shown in \cite{Peyre_SIIS2015} that the sequence $\{\textbf{P}^n\}_{n\ge0}$ generated by the above Dykstra algorithm converges to the solution of the problem  (\ref{eqn:BregmanProblem}).

When taking $g(\cdot)$ as the entropy function, the corresponding Bregman divergence $B_{g}(\textbf{P}|\textbf{K}) = \text{KL}(\textbf{P}|\textbf{K})$ becomes the KL divergence. In this case, if $\hat{h}_1$ and $\hat{h}_2$ in (\ref{eqn:BregmanProblem}) are of the special form as
\[
\hat{h}_1(\textbf{P}) = h_1(\textbf{P}\mathbbm{1}),\ \hat{h}_2(\textbf{P}) = h_2(\textbf{P}^T\mathbbm{1}),
\]
then the problem (\ref{eqn:BregmanProblem}) reduces to
\begin{align}\label{eqn:KLdiv2}
\min_{\textbf{P}} \text{KL}(\textbf{P}|\textbf{K}) + h_1(\textbf{P}\mathbbm{1}) + h_2(\textbf{P}^T\mathbbm{1}).
\end{align}
which is consistent with the sOT formulation  (\ref{eqn:KLUUOT_formII}), after dividing $\epsilon$ over all terms.

In this case, the optimal solution $\textbf{P}$ has the following decomposition
\begin{align}\label{eqn:DiagonalDecomp}
\textbf{P} = \text{diag}(\textbf{u}) \textbf{K} \text{diag}(\textbf{v}),
\end{align}
which is a diagonal scaling of the initial Gibbs kernel $\textbf{K}$, the same as the optimal solution for the regularized OT problem (which corresponds to $h_1(x) = \iota_{\{x=\mathbf{a}\}}(x)$ and $h_2(x) = \iota_{\{x=\mathbf{b}\}}(x)$ in (\ref{eqn:KLdiv2})). Indeed, this decomposition (\ref{eqn:DiagonalDecomp}) is not only holds for the optimal $\textbf{P}$, but it also holds for each iterate $\textbf{P}^n$ generated by Dykstra's algorithm for KL divergence. Therefore we assume that $\textbf{P}^n = \text{diag}(\textbf{u}^n) \textbf{K} \text{diag}(\textbf{v}^n)$. Then the Dykstra's algorithm can be written in an implementable form given as follows \cite{Chizat_2018}:

\begin{tcolorbox}
Dykstra algorithm for KL divergence (implementable form)   \\

Input:  $\textbf{u}^0 = \textbf{v}^0 = \mathbbm{1}$; \\
General step: for any $k=0,1,2,\cdots$ execute the following steps:
\begin{align}
\textbf{u}^{2k+1} &= \frac{\text{prox}_{h_1}^{\text{KL}}\big( \textbf{K} \textbf{v}^{2k} \big)}{\textbf{K} \textbf{v}^{2k}}, \ \textbf{v}^{2k+1} = \textbf{v}^{2k}; \\
\textbf{v}^{2k+2} &= \frac{\text{prox}_{h_2}^{\text{KL}}\big( \textbf{K}^T \textbf{u}^{2k+1} \big)}{\textbf{K}^T \textbf{u}^{2k+1}}, \ \textbf{u}^{2k+2} = \textbf{u}^{2k+1}.
\end{align}
\end{tcolorbox}
Note that in some literatures, this implementable form of the Dykstra's algorithm for KL divergence is called {\it{generalized Sinkhorn iteration}} for problem (\ref{eqn:KLdiv2}).

\subsection{Dykstra Algorithm for entropy regularized sOT problem}

The entropy regularized sOT problem (\ref{eqn:KLUUOT_formII}) is a special case of the KL divergence problem (\ref{eqn:KLdiv2}) by taking 
\begin{align}\label{eqn:prox_OptionI}
\quad \frac{1}{\epsilon} h_1(\textbf{P}\mathbbm{1}) = \gamma \|\textbf{a}-\textbf{P}\mathbbm{1}\|_1 + \iota_{[0,\textbf{a}]}(\textbf{P}\mathbbm{1}), \quad \frac{1}{\epsilon}h_2(\textbf{P}^T\mathbbm{1}) =  \gamma \|\textbf{b}-\textbf{P}^T\mathbbm{1}\|_1 + \iota_{[0,\textbf{b}]}(\textbf{P}^T\mathbbm{1}).
\end{align}
Indeed, $h_1/\epsilon$ (and $h_2/\epsilon$, respectively) can be viewed as a regularization term to render $\boldsymbol\upmu = \textbf{a} - \textbf{P}\mathbbm{1}$ ($\boldsymbol\upnu = \textbf{b} - \textbf{P}^T\mathbbm{1}$, respectively) as small as possible but within the range $[0,\textbf{a}]$ (and $[0,\textbf{b}]$, respectively). In this case, the implementation of Dykstra algorithm depends on the form of the proximal operator of $\|\cdot\|_1$ with respect to the KL divergence, which is given in the following lemma.
\begin{lemma}\label{lemma:prox_OptionI}
Let $h_i(\cdot) = \gamma \|\textbf{a}_i-\cdot\|_1 + \iota_{[0,\textbf{a}_i]}(\cdot), i = 1, 2$ with $\textbf{a}_1 = \textbf{a}, \textbf{a}_2 = \textbf{b}$, then the proximal operator of $h_i$ with respect the \emph{KL} divergence is given as
\begin{align}
\emph{prox}_{h_i/\epsilon}^{\emph{KL}} (\textbf{q}) = \min\{e^{\gamma/\epsilon}\textbf{q}, \textbf{a}_i\}, \ i = 1, 2.
\end{align}
\end{lemma}
\begin{proof}
By definition of the proximal operator with respect to the KL divergence, we have
\begin{align*}
\text{prox}_{(\gamma/\epsilon)\|\textbf{a}_i-\cdot\| + \iota_{[0,\textbf{a}_i]}(\cdot)}^{\text{KL}}(\textbf{q}) &= 
\underset{\textbf{p}\in[0,\textbf{a}_i]}{\mathrm{argmin}}\ \text{KL}(\textbf{p}|\textbf{q}) + (\gamma/\epsilon)\|\textbf{a}_i-\textbf{p}\|
\end{align*}
If $\textbf{a}_i\le \textbf{q}$, both $\text{KL}(\cdot|\textbf{q})$ and $\|\textbf{a}_i-\cdot\|$ decrease over domain $[0,\textbf{a}_i]$, so the minimum is attained at $\textbf{p} = \textbf{a}_i$; if $\textbf{a}_i \ge \textbf{q}$, taking the derivative of $\text{KL}(\textbf{p}|\textbf{q}) + (\gamma/\epsilon)\|\textbf{a}_i-\textbf{p}\|$ with respect to $\textbf{p}$ and set it to zero, we find $\textbf{p} = \min\{e^{\gamma/\epsilon}\textbf{q}, \textbf{a}_i\}$. Combining both two cases yields the result.
\end{proof}

Inserting Lemma \ref{lemma:prox_OptionI} to the Dykstra algorithm for KL divergence,  we obtain the generalized Sinkhorn iteration for entropy regularized sOT problem with the regularization terms in (\ref{eqn:prox_OptionI}):
\begin{tcolorbox}

Generalized Sinkhorn algorithm for sOT  \\

Input: $\textbf{u}^0 = \textbf{v}^0 = \mathbbm{1}$; \\
General step: for any $k=0,1,2,\cdots$ execute the following steps:
\begin{align}
\textbf{u}^{2k+1} &=  \frac{\min\{ e^{\gamma/\epsilon}\textbf{K}\textbf{v}^{2k}, \textbf{a}\}}{\textbf{K} \textbf{v}^{2k}} = \min \left\{ e^{\frac{\gamma}{\epsilon}}\mathbbm{1}, \frac{\textbf{a}}{\textbf{K}\textbf{v}^{2k}} \right\}, \ \textbf{v}^{2k+1} = \textbf{v}^{2k}; \\
\textbf{v}^{2k+2} &=  \frac{\min\{ e^{\gamma/\epsilon}\textbf{K}^T\textbf{u}^{2k+1}, \textbf{b}\}}{\textbf{K}^T \textbf{u}^{2k+1}} = \min \left\{ e^{\frac{\gamma}{\epsilon}}\mathbbm{1}, \frac{\textbf{b}}{\textbf{K}^T\textbf{u}^{2k+1}} \right\}, \ \textbf{u}^{2k+2} = \textbf{u}^{2k+1}.
\end{align}

\end{tcolorbox}

\section{sOT barycenter}

Given a set $\{\textbf{b}_j\}_{j=1}^J$  of unbalanced marginal densities $\textbf{b}_j \in \mathbb{R}_{+}^{m}$ and a weight $\lambda = (\lambda_1,\cdots, \lambda_J) \in \text{int}(\Delta_J)$, it is of practical interest to compute the weighted sOT barycenter of  $\{\textbf{b}_j\}_{j=1}^J$. This problem can be viewed as the generalization of the standard Wasserstein barycenter problem studied in \cite{CuturiDoucet_2014}.

We define the sOT barycenter problem in a similar manner as that for sOT. Let $\boldsymbol\upnu = (\boldsymbol\upnu_j)_{j=1}^J\in(\mathbb{R}_+^m)^J$ denote the blocked marginal measure, and $\|\boldsymbol\upnu\|_1 = \sum_{j} \|\boldsymbol\upnu_j\|_1$. We define two sets, one of which is for the feasible blocked marginal density $\boldsymbol\upnu$, and the other of which is for the feasible $\boldsymbol\upnu$ with minimal 1-norm:
\begin{align*}
&\mathcal{F}= \{\boldsymbol\upnu: \|\mathbf{b}_1-\boldsymbol\upnu_1\|_1=\cdots=\|\mathbf{b}_J-\boldsymbol\upnu_J\|_1, \text{ and } \exists\ \mathbf{P}_j\in\mathbf{U}(\mathbf{a},\mathbf{b}_j-\boldsymbol\upnu_j), \text{ such that  } \langle \mathbf{P}_j, \mathbf{C} \rangle < \infty \text{ for some } \mathbf{a}\}, \\
&\mathcal{G}= \underset{\boldsymbol\upnu \in\mathcal{F}}{\mathrm{argmin}}  \|\boldsymbol\upnu\|_1.
\end{align*}
Then the sOT barycenter problem is defined as:
\begin{align}\label{eqn:UUOTBarycenter01}
\min_{\boldsymbol\upnu\in\mathcal{G}} \min_{\mathbf{P}_j\in\mathbf{U}(\mathbf{a},\mathbf{b}_j-\boldsymbol\upnu_j)} \sum_{j=1}^J \lambda_j \langle \mathbf{P}_j, \mathbf{C} \rangle.
\end{align}

Similarly as the equivalence between various sOT formulations, we can show that (\ref{eqn:UUOTBarycenter01}) is equivalent to 
\begin{align}\label{eqn:UUOTBarycenter02}
\min_{\substack{\boldsymbol\upnu \in\mathcal{F} \\ \mathbf{P}_j\in\mathbf{U}(\mathbf{a},\mathbf{b}_j-\boldsymbol\upnu_j)}} \sum_{j=1}^J \lambda_j \langle \mathbf{P}_j, \mathbf{C} \rangle + \frac{\gamma}{J} \|\boldsymbol\upnu\|_1
\end{align}
for sufficiently large $\gamma$. The equivalence between (\ref{eqn:UUOTBarycenter01}) and (\ref{eqn:UUOTBarycenter02}) is summarized in theorem \ref{lemma:UUOTbarycenter}.
\begin{theorem}\label{lemma:UUOTbarycenter}
Given a cost matrix $\mathbf{C}$ with $\infty$-pattern $\mathcal{P}_{\infty}(\mathbf{C})$, the two formulations for the sOT barycenter problems (\ref{eqn:UUOTBarycenter01}) and (\ref{eqn:UUOTBarycenter02}) are equivalent for sufficiently large $\gamma$.
\end{theorem}
\begin{proof}
The proof is similar to that of Lemma \ref{lemma:equivalence}. Starting from an optimal $\boldsymbol\upnu^{\text{opt}}= (\boldsymbol\upnu^{\text{opt}}_j)_j$ for the formulation (\ref{eqn:UUOTBarycenter01}) and a corresponding optimal plan $\mathbf{P}^{\bar{\theta}, *}=(\mathbf{P}^{\bar{\theta}, *}_j)_j$ in which $\bar{\theta} = \| \mathbf{b}_1 - \boldsymbol\upnu^{\text{opt}}_1\|_1 = \cdots = \| \mathbf{b}_J - \boldsymbol\upnu^{\text{opt}}_J\|_1$, and taking any nonnegative and feasible $\theta<\bar{\theta}$, if we can prove that for any plan $\mathbf{P}^{\theta, *}$ defined as
\begin{align}\label{eqn:barycenter_temp00}
\mathbf{P}^{\theta, *} =  \underset{ \textbf{P} } {\text{argmin}} \left\{ \sum_{j=1}^J \lambda_j \langle \textbf{P}_j, \textbf{C} \rangle: \textbf{P}_j \in\mathbf{U}(= \textbf{a}, \le \textbf{b}_j), \mathbf{P}_1\mathbbm{1}=\cdots=\mathbf{P}_J\mathbbm{1}=\mathbf{a},   \langle \textbf{P}_j, \mathbbm{1}\rangle = \theta, \forall j \right\},
\end{align}
one has that 
\begin{align}\label{eqn:barycenter_temp01}
\sum_{j=1}^J \lambda_j \langle \textbf{P}^{\bar{\theta},*}_j - \textbf{P}^{\theta, *}_j, \textbf{C} \rangle \le   \gamma \sum_{j=1}^J \lambda_j \langle \textbf{P}^{\bar{\theta},*}_j - \textbf{P}^{\theta, *}_j, \mathbbm{1} \rangle, \quad \text{for some }\gamma>0,
\end{align}
then for any $\boldsymbol\upnu = (\boldsymbol\upnu_j)_j$ such that $\theta = \| \mathbf{b}_1 - \boldsymbol\upnu_1\|_1 = \cdots = \| \mathbf{b}_J - \boldsymbol\upnu_J\|_1$, it implies that
\begin{align*}
\|\boldsymbol\upnu\|_1  -  \|\boldsymbol\upnu^{\text{opt}}\|_1  
= \langle \boldsymbol\upnu - \boldsymbol\upnu^{\text{opt}}, \mathbbm{1} \rangle 
= J \sum_{j=1}^J \lambda_j \left\langle \textbf{P}^{\bar{\theta},*}_j - \textbf{P}^{\theta,*}_j, \mathbbm{1} \right\rangle
\ge  J\gamma^{-1} \sum_{j=1}^J \lambda_j \left\langle \textbf{P}^{\bar{\theta},*}_j - \textbf{P}^{\theta,*}_j, \mathbf{C} \right\rangle,
\end{align*}
leading to
\[
J\sum_{j=1}^J \lambda_j  \left\langle \mathbf{P}^{\bar{\theta},*}_j, \mathbf{C} \right\rangle + \gamma \|\boldsymbol\upnu^{\text{opt}}\|_1
\le J\sum_{j=1}^J \lambda_j \left\langle \mathbf{P}^{\theta,*}_j, \mathbf{C} \right\rangle + \gamma \|\boldsymbol\upnu\|_1,
\]
which implies the optimality of $(\boldsymbol\upnu^{\text{opt}},\mathbf{P}^{\bar{\theta},*})$ for (\ref{eqn:UUOTBarycenter02}).

Now we prove that there exists a $\gamma>0$ such that (\ref{eqn:barycenter_temp01}) holds. Using lemma \ref{lemma:polytope} and taking the bounded convex polyhedron $C$ to be the set of $\mathbf{P} = (\mathbf{P}_j)_j$ defined by the constraints
\begin{align*}
\begin{cases}
\mathbf{P}_j^T\mathbbm{1}\le \mathbf{b}_j,  \mathbf{P}_j\ge0, \quad j=1:J \\
\mathbf{P}_j\mathbbm{1} = \mathbf{P}_{j+1}\mathbbm{1}, \quad j=1:J-1 \\
\mathbf{P}_j\ge0, \quad j=1:J \\
(\mathbf{P}_j)_{kl}=0, \quad (k,l)\in\mathcal{P}_{\infty}(\mathbf{C}), \quad j=1:J \\
\end{cases}
\end{align*}
and 
\[
C_{\theta} = \{\mathbf{P}\in C: \langle \mathbf{P}, \mathbbm{1} \rangle: = \sum_{j=1}^J \langle \mathbf{P}_j , \mathbbm{1} \rangle = J\theta  \},
\]
we know that there exists a critical $\theta_0$ such that for any given $\theta \in (\theta_0,\bar{\theta}]$ and $\mathbf{P}^{\theta,*}$ defined in (\ref{eqn:barycenter_temp00}), we can find a feasible $\mathbf{P}^{\bar{\theta}}$ satisfying $\langle \mathbf{P}^{\bar{\theta}}, \mathbbm{1} \rangle = J\bar{\theta}$, such that
\begin{align}
\| \mathbf{P}^{\bar{\theta}} - \mathbf{P}^{\theta,*} \|_1 \le J\eta|\bar{\theta}-\theta|.
\end{align}
Then 
\begin{align}
\sum_{j=1}^J \lambda_j \left\langle \textbf{P}^{\bar{\theta},*}_j - \textbf{P}^{\theta,*}_j, \textbf{C} \right\rangle & \le
\sum_{j=1}^J \lambda_j \left\langle \textbf{P}^{\bar{\theta}}_j - \textbf{P}^{\theta,*}_j, \textbf{C} \right\rangle \nonumber\\
& \le \frac{1}{J}\| \mathbf{P}_{\bar{\theta}} - \mathbf{P}_{\theta}^* \|_1 \cdot \|\mathbf{C}\|_{\infty} \nonumber \\
& \le \|\mathbf{C}\|_{\infty}\cdot\eta|\bar{\theta}-\theta| \nonumber \\
& = \|\mathbf{C}\|_{\infty} \cdot \eta  \sum_{j=1}^J \lambda_j  \left\langle \textbf{P}^{\bar{\theta},*}_j - \textbf{P}^{\theta,*}_j, \mathbbm{1} \right\rangle. \label{eqn:barycenter_temp06}
\end{align}
On the other hand, for any feasible $\theta\le\theta_0$ and $\mathbf{P}^{\theta,*}$ defined in (\ref{eqn:barycenter_temp00}), we simply have
\begin{align}
\sum_{j=1}^J \lambda_j \left\langle \textbf{P}^{\bar{\theta},*}_j - \textbf{P}^{\theta,*}_j, \textbf{C} \right\rangle & 
= \sum_{j=1}^J \lambda_j  \langle \textbf{P}^{\bar{\theta},*}_j, \textbf{C} \rangle -
   \sum_{j=1}^J \lambda_j  \langle \textbf{P}^{\theta,*}_j, \textbf{C} \rangle  \nonumber\\
& \le \|\mathbf{C}\|_{\infty}\cdot \bar{\theta} \nonumber\\
&=\|\mathbf{C}\|_{\infty}\cdot \frac{\bar{\theta}}{\bar{\theta}-\theta}(\bar{\theta}-\theta) \nonumber\\
&\le \|\mathbf{C}\|_{\infty}\cdot \frac{\bar{\theta}}{\bar{\theta}-\theta_0}(\bar{\theta}-\theta) \nonumber \\
&= \|\mathbf{C}\|_{\infty}\cdot \frac{\bar{\theta}}{\bar{\theta}-\theta_0} \eta  \sum_{j=1}^J \lambda_j  \left\langle \textbf{P}^{\bar{\theta},*}_j - \textbf{P}^{\theta,*}_j, \mathbbm{1} \right\rangle. \label{eqn:barycenter_temp07}
\end{align}
Finally, combining the inequalities (\ref{eqn:barycenter_temp06}) and (\ref{eqn:barycenter_temp07}) and taking $\gamma = \max\{\eta, \bar{\theta}/(\bar{\theta}-\theta_0)\}\cdot \|\mathbf{C}\|_{\infty}$, we prove the inequality (\ref{eqn:barycenter_temp01}), and therefore the equivalence between two formulations.

\end{proof}

Note that the feasible set $\mathcal{F}$ is always non-empty, the formulation (\ref{eqn:UUOTBarycenter02}) can be recast into the form
\begin{align}\label{eqn:UUOTBarycenter03}
\min_{\substack{(\mathbf{P}_j, \mathbf{a}) \\ \mathbf{P}_j\in\mathbf{U}( = \mathbf{a}, \le\mathbf{b}_j)}} \sum_{j=1}^J \lambda_j \langle \mathbf{P}_j, \mathbf{C} \rangle + \gamma \sum_{j=1}^J \| \mathbf{b}_j - \mathbf{P}_j^T\mathbbm{1}\|_1.
\end{align}
Here we replace $\frac{\gamma}{J}$ by $\gamma$ in the equivalent formulation for the sake of simple notation.


\subsection{Entropic regularization of the sOT barycenter problem}

In this section, we consider the entropic regularization for the weighted sOT barycenter problem. To this end, we introduce the following notations:
\begin{align}
&\text{KL}_{\lambda}(\textbf{P}|\textbf{Q}): = \sum_{j=1}^J \lambda_j \text{KL}( \mathbf{P}_j| \mathbf{Q}_j),\quad \text{where\ } \textbf{P} = (\mathbf{P}_j)_{j}\in(\mathbb{R}_+^{n\times m})^J, \textbf{Q} = (\mathbf{Q}_j)_{j}\in(\mathbb{R}_{++}^{n\times m})^J, \label{eqn:KL_lambda}\\
&\hat{h}_1(\mathbf{P}) = \iota_{\mathcal{D}}(\mathbf{P}_{1}\mathbbm{1}, \cdots,  \mathbf{P}_{J}\mathbbm{1}),  \quad \mathcal{D}:= \left\{(\mathbf{p}_1, \cdots, \mathbf{p}_J)\in (\mathbb{R}_+^m)^J : \mathbf{p}_1 = \cdots = \mathbf{p}_J  \right\},\label{eqn:h_hat1}\\
&\hat{h}_2(\mathbf{Q}) = \gamma \| \mathbf{b}-\mathbf{Q}^T\mathbbm{1}\|_1 + \iota_{[0,\textbf{b}]}(\textbf{Q}^T\mathbbm{1}) := \sum_{j=1}^J \Big(\gamma\| \mathbf{b}_j-\mathbf{Q}_j^T\mathbbm{1}\|_1 + \iota_{[0,\mathbf{b}_j]}(\mathbf{Q}_j^T\mathbbm{1}) \Big). \label{eqn:h_hat2} \\
&h_1(\mathbf{p}) =  \iota_{\mathcal{D}}(\mathbf{p}_{1}, \cdots,  \mathbf{p}_{J}), \quad \text{for\ } \textbf{p} = (\mathbf{p}_j)_{j}, \label{eqn:h_1}\\
&h_2(\mathbf{q}) = \gamma \| \mathbf{b}-\textbf{q}\|_1 + \iota_{[0,\textbf{b}]}(\textbf{q}) = \sum_{j=1}^J \Big(\gamma\| \mathbf{b}_j-\mathbf{q}_j\|_1 + \iota_{[0,\mathbf{b}_j]}(\mathbf{q}_j) \Big), \quad \text{for\ } \textbf{q} = (\mathbf{q}_j)_{j=1}^n.  \label{eqn:h_2}
\end{align}
Note that $\hat{h}_i(\mathbf{P}) = h_i(\mathbf{P}\mathbbm{1}), i=1,2.$
With the above notations, we can formulate the entropic regularized sOT barycenter problem as
\begin{align}
\min_{\mathbf{P} \in (\mathbb{R}_{+}^{n\times m})^J}   \sum_{j=1}^J \lambda_j \Big(\langle \mathbf{P}_j, \mathbf{C} \rangle - \epsilon  H(\mathbf{P}_j) \Big) + \hat{h}_1(\mathbf{P}) +  \hat{h}_2(\mathbf{P}),
\end{align}
or equivalently in terms of the KL divergence, 
\begin{align}\label{eqn:UUOTBarycenter04}
\min_{\mathbf{P} \in (\mathbb{R}_{+}^{n\times m})^J}   \text{KL}_{\lambda}(\mathbf{P}|\mathbf{K}) + \frac{1}{\epsilon}\hat{h}_1(\mathbf{P}) +  \frac{1}{\epsilon} \hat{h}_2(\mathbf{P}),
\end{align}
where $\mathbf{K} = (\mathbf{K}_j)_{j=1}^J$ with $\mathbf{K}_j = e^{-\mathbf{C}/\epsilon}, j=1,\cdots,J$. 

To solve the entropic regularized sOT barycenter problem (\ref{eqn:UUOTBarycenter04}), we adopt the generic diagonal scaling algorithm introduced in \cite{Peyre_SIIS2015} (also see \cite{Chizat_2018}),  in which each iterate $\mathbf{P}$ has the diagonal scaling decomposition
\begin{align}
\mathbf{P}^{(n)} = (\mathbf{P}^{(n)} _j)_j = \left( \text{diag}(\mathbf{u}^{(n)} _j)\mathbf{K}\text{diag}(\mathbf{v}^{(n)} _j) \right)_j.
\end{align}
With a slight abuse of notation, we denote, consistent with $\mathbf{P} = (\mathbf{P}_j)_j$,
\[
\mathbf{u} = (\mathbf{u}_1, \cdots, \mathbf{u}_J) \in (\mathbb{R}^n)^J, \quad \mathbf{v} = (\mathbf{v}_1, \cdots, \mathbf{v}_J) \in (\mathbb{R}^m)^J.
\]
Then the diagonal scaling algorithm reads
\begin{tcolorbox}
Input:  $\mathbf{u}^{(0)} = \mathbf{v}^{(0)}  = \mathbbm{1}$; \\
General step: for any $n=0,1,2,\cdots$ execute the following steps:
\begin{align}
\mathbf{u}^{(2n+1)}_j &= \frac{ \left[ \text{prox}_{h_1}^{\text{KL}_{\lambda}}\big( \mathbf{K} \mathbf{v}^{(2n)} \big) \right]_j}{\mathbf{K} \mathbf{v}_j^{(2n)}}, \ \mathbf{v}_j^{(2n+1)} = \mathbf{v}_j^{(2n)}, \quad \forall j; \\
\mathbf{v}^{(2n+2)}_j &= \frac{ \left[ \text{prox}_{h_2}^{\text{KL}_{\lambda}}\big( \mathbf{K}^T \mathbf{u}^{(2n)} \big) \right]_j}{\mathbf{K}^T \mathbf{u}_j^{(2n)}}, \ \mathbf{u}_j^{(2n+2)} = \mathbf{u}_j^{(2n+1)}, \quad \forall j.
\end{align}
\end{tcolorbox}

Note that one needs to compute the two proximal operator $\text{prox}_{h_1}^{\text{KL}_{\lambda}}$ and $\text{prox}_{h_2}^{\text{KL}_{\lambda}}$ for $h_1$ and $h_2$ defined in (\ref{eqn:h_1})-(\ref{eqn:h_2}) to implement the diagonal scaling algorithm. The following lemma shows that the two proximal operators for the KL divergence can be computed in closed form. The derivation is similar to that of Proposition 5.1 in \cite{Peyre_SIIS2015}, so we omit the details.

\begin{lemma}
For any $\mathbf{p} = (\mathbf{p}_j)_j\in(\mathbb{R}^n)^J$, and $h_1$ and $h_2$ defined in (\ref{eqn:h_1})-(\ref{eqn:h_2}), one has
\begin{align}\label{eqn:prox_h1h2}
 \left[ \emph{prox}_{h_1}^{\emph{KL}_{\lambda}}\big( \mathbf{p} \big) \right]_j = \mathbf{p}_1^{\lambda_1} \odot \cdots \odot \mathbf{p}_J^{\lambda_J}, \quad  
 \left[ \emph{prox}_{h_2}^{\emph{KL}_{\lambda}}\big( \mathbf{p} \big) \right]_j = \min \left\{ e^{\frac{\gamma}{\lambda_j\epsilon}}\mathbf{p}_j, \mathbf{b}_j \right\}. 
 \end{align}
Additionally, for any $\mathbf{P} = (\mathbf{P}_j)_j\in(\mathbb{R}^{n\times m})^J$, and $\hat{h}_1$ and $\hat{h}_2$ defined in (\ref{eqn:h_hat1})-(\ref{eqn:h_hat2}), the two proximal operators $\emph{prox}_{\hat{h}_1}^{\emph{KL}_{\lambda}}$ and $\emph{prox}_{\hat{h}_2}^{\emph{KL}_{\lambda}}$ are related to (\ref{eqn:prox_h1h2}) as
\begin{align}
\left[\emph{prox}_{\hat{h}_1}^{\emph{KL}_{\lambda}}(\mathbf{P})\right]_j = \emph{diag}\left( \frac{  \left[ \emph{prox}_{h_1}^{\emph{KL}_{\lambda}}\big( \mathbf{P}\mathbbm{1} \big) \right]_j  }{\mathbf{P}_j\mathbbm{1}}  \right) \mathbf{P}_j, \ 
\left[\emph{prox}_{\hat{h}_2}^{\emph{KL}_{\lambda}}(\mathbf{P})\right]_j = \mathbf{P}_j \emph{diag}\left( \frac{  \left[ \emph{prox}_{h_2}^{\emph{KL}_{\lambda}}\big( \mathbf{P}^T\mathbbm{1} \big) \right]_j  }{\mathbf{P}^T_j\mathbbm{1}}  \right), 
\end{align}
in which
\[
\mathbf{P}\mathbbm{1} := (\mathbf{P}_1\mathbbm{1},\cdots \mathbf{P}_J\mathbbm{1}), \quad 
\mathbf{P}^T\mathbbm{1} := (\mathbf{P}^T_1\mathbbm{1},\cdots \mathbf{P}^T_J\mathbbm{1}).
\]
\end{lemma}

With the proximal operators computed in (\ref{eqn:prox_h1h2}), the diagonal scaling algorithm in a directly implementable form
becomes
\begin{tcolorbox}
Input:  $\mathbf{u}^{(0)} = \mathbf{v}^{(0)}  = \mathbbm{1}$; \\
General step: for any $n=0,1,2,\cdots$ execute the following steps:
\begin{align}
&\mathbf{u}_j^{(n+1)} = \frac{\mathbf{a}^{(n)}}{\mathbf{K}\mathbf{v}_j^{(n)}}, \quad j=1:J, \quad \text{where } \mathbf{a}^{(n)} = \prod_j \left(\mathbf{K}\mathbf{v}_j^{(n)}\right)^{\lambda_j},  \\
&\mathbf{v}_j^{(n+1)} = \min \left\{ \frac{\mathbf{b}_j}{\mathbf{K}^T\mathbf{u}_j^{(n)}}, e^{\frac{\gamma}{\lambda_j\epsilon}} \right\}, \quad j=1:J.
\end{align}
\end{tcolorbox}

\subsection{Log-domain implementation}

One drawback for the diagonal scaling algorithm (Sinkhorn algorithm) is that it suffers from numerical overflow when the regularization parameter $\epsilon$ is too small compared to the entries of the cost matrix $\mathbf{C}$. This drawback is even more severe for the sOT problem as it will cause some entries of $\mathbf{K} = e^{-\mathbf{C}/\epsilon}$ being regarded as zero due to the numerical overflow, even they should not. In other words, having more zero entries in $\mathbf{K}$ because of the smallness of $\epsilon$ will change the 0-pattern of $\mathbf{K}$ and consequently the $\infty$-pattern of $\mathbf{C}$. Therefore it is necessary to implement the diagonal scaling algorithm for sOT barycenter problem in the log-domain.

Using the log-sum-exp stabilization trick for the soft-minimization, and noting the primal-dual relation
\[
(\mathbf{u}^{(n)},\mathbf{v}^{(n)}) = \left(e^{\mathbf{f}^{(n)}/\epsilon}, e^{\mathbf{g}^{(n)}/\epsilon} \right),
\] 
the log-domain implementation for the diagonal scaling algorithm reads
\begin{tcolorbox}
Input:  $\mathbf{f}^{(0)} = \mathbf{g}^{(0)}  = 0$; \\
General step: for any $n=0,1,2,\cdots$ execute the following steps:
\begin{align}
&\mathbf{f}_j^{(n+1)} = \sum_{i} \lambda_i \left[ \epsilon \log \left( e^{(\mathbf{f}_i^{(n)}\oplus \mathbf{g}_i^{(n)} - \mathbf{C})/\epsilon}  \mathbbm{1} \right) - \mathbf{f}_i^{(n)} \right ] - \left[ \epsilon \log \left( e^{(\mathbf{f}_j^{(n)}\oplus \mathbf{g}_j^{(n)} - \mathbf{C})/\epsilon}  \mathbbm{1} \right) - \mathbf{f}_j^{(n)} \right ] ,  \\
&\mathbf{g}_j^{(n+1)} = \min \left\{ \epsilon\log(\mathbf{b}_j) -  \epsilon \log \left( e^{(\mathbf{f}_j^{(n+1)}\oplus \mathbf{g}_j^{(n)} - \mathbf{C})^T/\epsilon}  \mathbbm{1} \right) + \mathbf{g}_j^{(n)}  , \frac{\gamma}{\lambda_j} \right\}.
\end{align}
for all $j=1:J$.
\end{tcolorbox}

\subsection{Special case in which $\lambda \in \partial(\Delta_J)$}

In this subsection, we point out an important difference between the standard OT barycenter problem and the sOT one. For the sake of simplicity, we take $J = 2$. 

Note that when $\lambda = (0,1)$, the OT barycenter problem degenerates to the standard OT problem. More precisely
\[
\min_{\mathbf{a}}\lambda_1 L_{\text{OT}}(\mathbf{a},\mathbf{b}_1) + \lambda_2 L_{\text{OT}}(\mathbf{a},\mathbf{b}_2),
\]
reduces to 
\[
\min_{\mathbf{a}}  L_{\text{OT}} (\mathbf{a},\mathbf{b}_2),
\]
which leads to $\mathbf{a} = \mathbf{b}_2$ and $\mathbf{P}_2 = \text{diag}(\mathbf{b}_2)$. Then $\mathbf{P}_1$ is determined by the standard OT 
\[
L_{\text{OT}}(\mathbf{a},\mathbf{b}_1) = \min_{\mathbf{P}_1 \in \mathbf{U}(\mathbf{a},\mathbf{b}_1)} \langle \mathbf{P}_1,\mathbf{C} \rangle.
\]
Additionally, the entropic OT barycenter problem reduces to the entropic OT problem. In other words, 
\[
\min_{\mathbf{a}}\lambda_1 L^{\epsilon}_{\text{OT}}(\mathbf{a},\mathbf{b}_1) + \lambda_2 L^{\epsilon}_{\text{OT}}(\mathbf{a},\mathbf{b}_2),
\]
reduces to 
\[
\min_{\mathbf{a}}  L^{\epsilon}_{\text{OT}} (\mathbf{a},\mathbf{b}_2),
\]
which leads to $\mathbf{a} = \frac{\mathbf{K}\mathbf{b}_2}{\mathbf{K}^T\mathbbm{1}}$ and $\mathbf{P}_2 = \mathbf{K} \text{diag}\left(\frac{\mathbf{b}_2}{\mathbf{K}^T\mathbbm{1}}\right)$. Then $\mathbf{P}_1$ is determined by the entropic OT
\[
L_{\text{OT}}^{\epsilon}(\mathbf{a},\mathbf{b}_1) = \min_{\mathbf{P}_1 \in \mathbf{U}(\mathbf{a},\mathbf{b}_1)} \epsilon \text{KL}(\mathbf{P}_1 | \mathbf{K}).
\]

However, such degeneration does not apply to the sOT barycenter problem,
\[
\lim_{\lambda_1 \rightarrow 0} \Big( \min_{\mathbf{a}}\lambda_1 L_{\text{sOT}}(\mathbf{a},\mathbf{b}_1) + \lambda_2 L_{\text{sOT}}(\mathbf{a},\mathbf{b}_2) \Big) \neq
\min_{\mathbf{a}}  L_{\text{sOT}} (\mathbf{a},\mathbf{b}_2) .
\]
To elucidate the idea, we take the cost matrix $\mathbf{C}$ as in (\ref{eqn:CostMatrix}) with $C_{\text{cut}} = 0.3$, namely, any mass can only be transported within the distance no longer than $C_{\text{cut}}$. We take $\mathbf{y}\in\mathbb{R}^{n+1}$ be a uniform mesh over $[0,1]$ with $h = 1/n$ being the mesh spacing, and let $\mathbf{b}_1=\mathbbm{1}$ (and $\mathbf{b}_2=\mathbbm{1}$, respectively) be uniform distribution on $\mathbf{y}$ compactly supported over $[0.1, 0.3]$ (and $[0.7,0.9]$, respectively). For any $\lambda = (\lambda_1, \lambda_2) \in \text{int}(\Delta_2)$, the sOT barycenter $\mathbf{a}$ must be the uniform distribution $\mathbf{a} = \mathbbm{1}$ on $\mathbf{y}$ compactly supported over $[0.4,0.6]$. The corresponding total cost is
\[
\lambda_1(0.3)^2 + \lambda_2 (0.3)^2 = (0.3)^2.
\]
This is because any other possible transport plan will cause some mass, even only a bit, being transported from either $\mathbf{b}_1$ or $\mathbf{b}_2$ to anywhere beyond $[0.4,0.6]$, resulting in an infinite cost. On the other hand, when $\lambda = (0, 1) \in \partial(\Delta_2)$, the sOT barycenter is determined by
\[
\min_{\mathbf{a}} L_{\text{sOT}}(\mathbf{a},\mathbf{b}_2). 
\]
Since it is unrelated to $L_{\text{sOT}}(\mathbf{a},\mathbf{b}_1)$, we can take $\mathbf{a} = \mathbf{b}_2$ such that the total cost equals zero. Hence the degeneration leads to a discontinuity for the total cost. To resolve this issue, we define the degenerate sOT problem for $\tilde{\lambda} \in \partial(\Delta_J)$ as the limiting problem when $\text{int}(\Delta_J) \ni \lambda \rightarrow \tilde{\lambda}$,
\[
 \min_{\mathbf{a}} \sum_{j} \tilde{\lambda}_j L_{\text{sOT}}(\mathbf{a},\mathbf{b}_j) : = 
\lim_{\text{int}(\Delta_J) \ni \lambda \rightarrow \tilde{\lambda}} \Big( \min_{\mathbf{a}} \sum_j \lambda_j L_{\text{sOT}}(\mathbf{a},\mathbf{b}_j) \Big).
\]

\subsection{Reverse and portion selection mechanism}

We can further perform theoretical analysis on the limiting behavior of the barycenter sOT problem. For the sake of simplicity, we still take $J = 2$. Let $\mathbf{y} = (y_i)_i\in\mathbb{R}^{n}$ be the uniform mesh over $[0,1-h]$ with $h = 1/n$. 

To begin with, we define a cumulative sum inequality as follows. For any two vectors $\mathbf{u}, \mathbf{v} \in\mathbb{R}^n$, we say $\mathbf{u}$ is cumulatively less than or equal to $\mathbf{v}$ and denote by $\mathbf{u} \le_{\text{C}} \mathbf{v}$ if they satisfy
\begin{align}\label{eqn:cumsum_def}
\sum_{j=1}^k u_j \le \sum_{j=1}^k v_j, k=1:n-1; \quad \text{and\ } \sum_{j=1}^n u_j = \sum_{j=1}^n v_j.
\end{align}
$\mathbf{u}$ is strictly cumulatively less than $\mathbf{v}$ and denote by $\mathbf{u} <_{\text{C}} \mathbf{v}$ if at least one inequality in (\ref{eqn:cumsum_def}) is strict for $k=1:n-1$.  We denote by $\mathbf{u} =_{\text{C}} \mathbf{v}$ if all the inequalities in (\ref{eqn:cumsum_def}) are equality. It is clear that $\mathbf{u} =_{\text{C}} \mathbf{v}$ if and only if $\mathbf{u} = \mathbf{v}$.

Let $T^t$ be a periodic shift operator for any periodic function $f(x)$ over $[0,1)$ such that $(T^t f)(x) = f(x+t)$. We take two nonnegative periodic functions $b_1(x),b_2(x)$ over $[0,1)$ with compact support $[0,0.2]$, and let them satisfy the cumulative sum inequality 
\begin{align}\label{eqn:cumsum_inequality_b1b2}
b_1|_{\mathbf{y} \cap [0,0.2]} \le_{\text{C}} b_2|_{\mathbf{y} \cap [0,0.2]}.
\end{align}
We define two marginal distributions $\mathbf{b}_1$ and $\mathbf{b}_2$ as
\begin{align}\label{eqn:b1b2}
\mathbf{b}_1 = (T^{0.1}b_1(x))|_{\mathbf{y}}, 
\quad
\mathbf{b}_2 = (T^{0.7}b_2(x))|_{\mathbf{y}}.
\end{align}
The cost matrix $\mathbf{C}$ is taken as
\begin{align}\label{eqn:CostMatrix0}
\mathbf{C}_{ij} =
\begin{cases}
| y_i - y_j |^2, \quad &\text{if} \ | y_i - y_j | \le 0.3, \\
\infty, \quad &\text{if} \ | y_i - y_j | > 0.3.
\end{cases}
\end{align}

In the following lemma, we characterize the marginal distributions which are of the same amount of mass as $\mathbf{b}_1$ and $\mathbf{b}_2$, and can be completely transported to $\mathbf{b}_1$ and $\mathbf{b}_2$ given the cost matrix $\mathbf{C}$ in (\ref{eqn:CostMatrix0}). 
\begin{lemma}\label{lemma:character_a}
Given marginal distributions $\mathbf{b}_1$ and $\mathbf{b}_2$ as in (\ref{eqn:b1b2}), a marginal distribution $\mathbf{a}$ can be completely transported to both $\mathbf{b}_1$ and $\mathbf{b}_2$ if and only if 
\begin{align}\label{eqn:a_character}
\mathbf{a} = (T^{0.4}a(x))|_{\mathbf{y}}
\end{align}
in which $a(x)$ is a nonnegative periodic function over $[0,1)$ with compact support $[0,0.2]$ and satisfies 
\begin{align}\label{eqn:cumsum_condition}
b_1(x)|_{\mathbf{y}\cap[0,0.2]} \le_{\emph{C}} a(x)|_{\mathbf{y}\cap[0,0.2]} \le_{\emph{C}} b_2(x)|_{\mathbf{y}\cap[0,0.2]}.
\end{align}
\end{lemma}
\begin{proof}
We only consider the case for $\mathbf{P}_2$. That for $\mathbf{P}_1$ is similar.

First, we show that if the cumulative sum inequality (\ref{eqn:cumsum_condition}) holds, then there exists some plan $\mathbf{P}_2$ to transport $\mathbf{a}$ competely to $\mathbf{b}_2$. Here we use the north-west corner rule \cite{MAL-073} to construct such a plan $\mathbf{P}_2$.  If $\mathbf{P}_2$ completely transports $\mathbf{a}$ to $\mathbf{b}_2$, then all the entries off the sub-matrix $\mathbf{Q} = \mathbf{P}_2(401:600,701:900) \in \mathbb{R}^{200\times 200}$ must be 0 due to the compact support of $\mathbf{a}$ and $\mathbf{b}_2$. Besides, since the entries of $\mathbf{C}$ are $\infty$ when $|\mathbf{y}_i-\mathbf{y}_j| > 0.3$, the strict upper triangular entries of $\mathbf{Q}$ must also be 0.  Now we apply the north-west corner rule  to determine the lower triangular (including diagonal) entries of $\mathbf{Q}$. More precisely, the rule starts by giving the highest possible value to $Q_{11} = (\mathbf{P}_2)_{401,701}$ by setting it to $\min\{(\mathbf{a})_{401},(\mathbf{b}_2)_{701}\}$. At each step, the entry $(\mathbf{P}_2)_{ij}$ is chosen to saturate either the $i$-th row constraint, $j$-th column constraint, or both if possible. The indices $i,j$ are then updated as follows: $i$ is incremented in the first case, $j$ is in the second, and both $i$ and $j$ are in the third case. The rule proceeds until $(\mathbf{P}_2)_{600,900}$ receives a value. On the other hand, using the second half of the cumulative sum inequality (\ref{eqn:cumsum_condition}), we have that each diagonal entry $Q_{kk}$ of $\mathbf{Q}$ must be chosen to saturate the corresponding row constraint or both the row and column constraints, but cannot saturate only the corresponding column constraint.  Since the total mass of $\mathbf{a}$ is equal to that of $\mathbf{b}_2$, the possible excessive mount of mass due to the row saturation is eventually transported by the last row of $\mathbf{Q}$. Hence mass is completely transported. 

Secondly, we show that if the cumulative sum inequality (\ref{eqn:cumsum_condition}) is violated, then no plan $\mathbf{P}_2$ can completely transport $\mathbf{a}$ to $\mathbf{b}_2$. Let $k$ be the smallest integer to break (\ref{eqn:cumsum_condition}),
\begin{align}\label{eqn:cumsum_violation}
\sum_{j=1}^{k} a(y_j) > \sum_{j=1}^{k} b_2(y_j).
\end{align}
Assume there exists a plan $\mathbf{P}_2$ which transports $\mathbf{a}$ completely to $\mathbf{b}_2$. We still denote $\mathbf{Q} = \mathbf{P}_2(401:600,701:900)$. Due to the compact support of $\mathbf{a},\mathbf{b}_2$ and the $\infty$ entries in $\mathbf{C}$, the nonzero entries of $\mathbf{P}_2$ can only lie in the lower triangular half (including the diagonal) of $\mathbf{Q}$. Consider the sub-matrix $\mathbf{Q}_k = \mathbf{Q}(1:k,1:k)$. On one hand, since $\mathbf{a}$ is completely transported, we have 
\[
\mathbf{Q}_k\mathbbm{1} = [ \mathbf{a}(y_1), \cdots, \mathbf{a}(y_k)]^T, \quad 
\mathbf{Q}_k^T\mathbbm{1} \le [ \mathbf{b}_2(y_1), \cdots, \mathbf{b}_2(y_k)]^T.
\]
However, $\sum_{i}(\mathbf{Q}_k\mathbbm{1})_i = \sum_{i,j} (\mathbf{Q}_k)_{ij}= \sum_j (\mathbf{Q}_k^T\mathbbm{1})_j$ leads to $\sum_{j=1}^{k} a(y_j) \le \sum_{j=1}^{k} b_2(y_j)$, causing a contradiction to inequality (\ref{eqn:cumsum_violation}).
\end{proof}

Then we have the following result regarding the sOT barycenter of $\mathbf{b}_1$ and $\mathbf{b}_2$.

\begin{theorem}\label{theorem:mirrority}
Let $\mathbf{b}_1$ and $\mathbf{b}_2$ be defined as in (\ref{eqn:b1b2}) with $b_1$ and $b_2$ satisfying (\ref{eqn:cumsum_inequality_b1b2}). Given cost matrix $\mathbf{C}$ defined in (\ref{eqn:CostMatrix0}), we have
\begin{align}
\lim_{(\lambda_1,\lambda_2)\rightarrow(1,0)} \underset{\mathbf{a}}{\mathrm{argmin}} \Big( \sum_{j=1}^2 \lambda_j L_{\emph{sOT}}(\mathbf{a},\mathbf{b}_j ; \mathbf{C})\Big) = (T^{0.4}b_2)|_{\mathbf{y}},\label{eqn:mirrority1}
\\
\lim_{(\lambda_1,\lambda_2)\rightarrow(0,1)} \underset{\mathbf{a}}{\mathrm{argmin}} \Big( \sum_{j=1}^2 \lambda_j L_{\emph{sOT}}(\mathbf{a},\mathbf{b}_j ; \mathbf{C}) \Big) = (T^{0.4}b_1)|_{\mathbf{y}}. \label{eqn:mirrority2}
\end{align}
\end{theorem}
\begin{proof}
We only prove the case for $(\lambda_1,\lambda_2)\rightarrow(0,1)$.The other case is similar. By the virtue of lemma \ref{lemma:UUOTbarycenter} and lemma \ref{lemma:character_a}, all the candidates for the sOT barycenter of $\mathbf{b}_1, \mathbf{b}_2$ are in the form of (\ref{eqn:a_character}) with condition (\ref{eqn:cumsum_condition}) satisfied. We denote by $\mathcal{S}_{\mathbf{a}}$ the collection of all these candidates.

In the limit $(\lambda_1,\lambda_2)\rightarrow(0,1)$, $\lambda_1 L_{\text{sOT}}(\mathbf{a},\mathbf{b}_1)$ becomes zero. We only need to seek an optimal $\mathbf{a}^*$ in $\mathcal{S}_{\mathbf{a}}$ to minimize $L_{\text{sOT}}(\mathbf{a},\mathbf{b}_2)$. Pick any $\mathbf{a} \ge_{\text{C}} \mathbf{b}_1$, and let $k$ be the smallest integer such that 
\begin{align}\label{eqn:cumsum_violation2}
\sum_{j=1}^{k} a(y_j) > \sum_{j=1}^{k} b_1(y_j).
\end{align}
Since $\sum_{j=1}^{n} a(y_j) = \sum_{j=1}^{n} b_1(y_j)$, there must be some integer $l>k$ such that $a(y_l)< b_1(y_l)$. Let $\mathbf{P}_{\mathbf{a}}^*$ be an optimal transport plan for $L_{\text{sOT}}(\mathbf{a},\mathbf{b}_2)$, and $(\mathbf{P}_{\mathbf{a}}^*)_{ki}$ be a nonzero entry on the $k$-th row. Then for a sufficiently small $\epsilon>0$, changing $(\mathbf{P}_{\mathbf{a}}^*)_{ki} \rightarrow (\mathbf{P}_{\mathbf{a}}^*)_{ki} - \epsilon$ and $(\mathbf{P}_{\mathbf{a}}^*)_{kl} \rightarrow (\mathbf{P}_{\mathbf{a}}^*)_{kl} + \epsilon$ causes a cost reduction. Therefore $\mathbf{a}$ is not the optimal candidate in $\mathcal{S}_{\mathbf{a}}$ unless $\mathbf{a} =_{\text{C}} \mathbf{b}_1$, that is, $\mathbf{a} = \mathbf{b}_1$.
\end{proof}

\begin{remark}
In the standard OT barycenter problem, we have
\begin{align}
\lim_{(\lambda_1,\lambda_2)\rightarrow(1,0)} \underset{\mathbf{a}}{\mathrm{argmin}} \Big( \sum_{j=1}^2 \lambda_j L_{\emph{OT}}(\mathbf{a},\mathbf{b}_j ; \mathbf{C}) \Big) = \mathbf{b}_1,
\quad
\lim_{(\lambda_1,\lambda_2)\rightarrow(0,1)} \underset{\mathbf{a}}{\mathrm{argmin}} \Big( \sum_{j=1}^2 \lambda_j L_{\emph{OT}}(\mathbf{a},\mathbf{b}_j ; \mathbf{C}) \Big) = \mathbf{b}_2.
\end{align}
However, the sOT barycenter problem, without considering the periodic translation of $\mathbf{b}_1$ and $\mathbf{b}_2$, gives exactly opposite results as shown in theorem \ref{theorem:mirrority}. We call the limits (\ref{eqn:mirrority1}) and (\ref{eqn:mirrority2}) a reverse and portion selection mechanism.
\end{remark}

\section{Numerical Results}

In this section, we will present several numerical experiments to  validate the proposed sOT problem.  Note that in Lemmas \ref{lemma:polytope} and \ref{lemma:equivalence}, we only prove the existence of $\lambda$ for the equivalence between the double minimization formulation (\ref{eqn:UUOT_formI}) and the single minimization formulation (\ref{eqn:UUOT_formII}), but there is no explicit evaluation of $\gamma$. For the numerical simulations, we will decide the value of $\gamma$ by the following procedure: taking several values of $\gamma = \gamma_1, \gamma_2, \cdots$ in ascending order and running the Dykstra solver for each value $\gamma_i$, until the total transported mass $\langle \mathbf{P}, \mathbbm{1} \rangle$ become unchanged (within certain accuracy) at some $\gamma_j$, then we choose $\gamma_j$ to be the value of $\gamma$ in the simulations. 

Figure \ref{fig:test_gamma} shows an example for this procedure. In this test, we take $\mathbf{C}$ as the truncated $L^2$ distance defined in (\ref{eqn:CostMatrix}) with $C_{cut} = 0.5$. The two marginal densities $\mathbf{a}$ and $\mathbf{b}$ are given as follows,
\begin{align}\label{eqn:Gaussian_equal}
\textbf{a} = \frac{1}{D}\left(e^{-\frac{(x-0.2)^2}{0.1^2}}+0.001\right)\bigg|_{h\mathbb{Z}\cap[0,1]}, \quad \textbf{b} = \frac{1}{D}\left(e^{-\frac{(x-0.8)^2}{0.1^2}}+0.001\right)\bigg|_{h\mathbb{Z}\cap[0,1]},
\end{align}
with $h = \frac{1}{200}$. We take $\gamma = 0, 0.1, 0.2,0.5,100$. For each value of $\gamma$, we calculate $\langle \mathbf{P}, \mathbbm{1} \rangle$. We note that when $\gamma = 0.5$ or larger, $\langle \mathbf{P}, \mathbbm{1} \rangle$ remains unchanged. Thereby we take $\gamma = 0.5$.

\begin{figure}[!htbp]
\centerline{
\includegraphics[width=60mm]{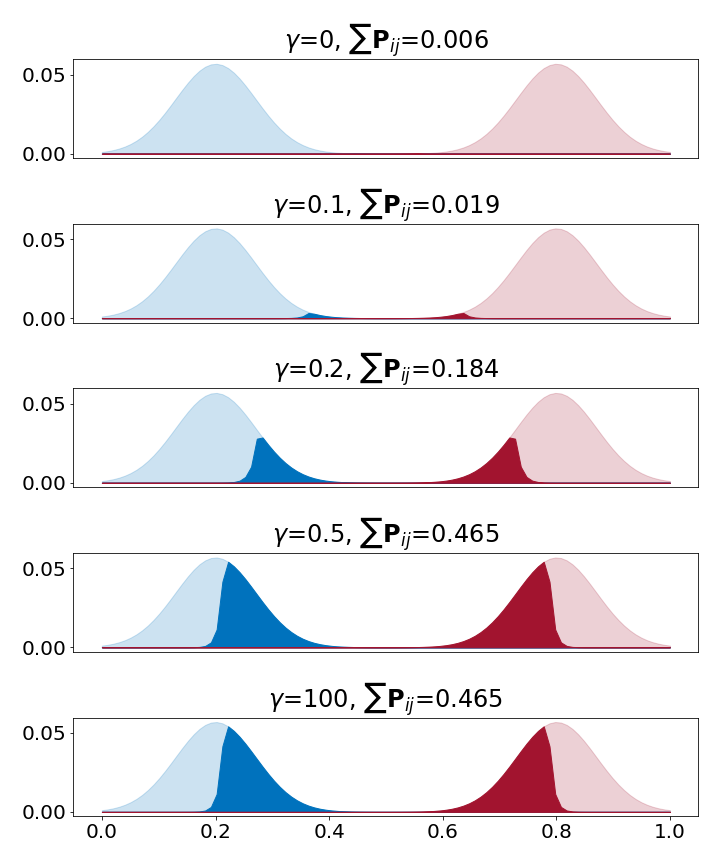}
 }
\caption{The schematic for taking the value of $\gamma$. We take several values of $\gamma$ in ascending order, run the Dykstra's algorithm, and evaluate the total transported mass $\langle \mathbf{P}, \mathbbm{1} \rangle$ until it becomes unchanged within certain accuracy. In this simulation, we take $\epsilon = 0.01$. Note that when $\gamma = 0$, there is still a tiny mass transported, which is due to the approximation of entry regularization.}
\label{fig:test_gamma}
\end{figure}

\subsection{Effect of $\epsilon$ (regulator weight) for entropy regularized sOT}

In our first example, the two marginal densities are taken as 1D discretized Gaussian distribution: 
\[
\textbf{a} = \frac{1}{2D}\left(e^{-\frac{(x-0.2)^2}{0.1^2}}+0.001\right)\bigg|_{h\mathbb{Z}\cap[0,1]}, \quad \textbf{b} = \frac{1}{D}\left(e^{-\frac{(x-0.8)^2}{0.1^2}}+0.001\right)\bigg|_{h\mathbb{Z}\cap[0,1]}, 
\]
with $h = \frac{1}{200}$. $D$ is a normalization constant such that $\|\mathbf{a} \|_1= \frac{1}{2}$ and $\|\mathbf{b}\|_1 = 1$. The cost matrix $\textbf{C} = (C_{ij})\in\mathbb{R}^{200\times 200}$ is taken as 
\begin{align}\label{eqn:CostMatrix}
C_{ij} =
\begin{cases}
|x_i - x_j|^2, \quad &\text{if} \ |x_i - x_j| \le C_\text{cut}, \\
\infty, \quad &\text{if} \ |x_i - x_j| > C_\text{cut}.
\end{cases}
\end{align}
with $\text{cutoff} = 0.7$. The value of $\gamma$ is taken to be $\gamma = 2$. By taking various values of $\epsilon = 1, 0.5, 0.1, 0.05, 0.025$ in Figure \ref{fig:test_eps}, we test the generalized Sinkhorn algorithm with $h_1$ and $h_2$ taken in the form of (\ref{eqn:prox_OptionI}). For all values of $\epsilon$, $\textbf{a}$ is completely transported. For the transported part of $\textbf{b}$, namely, the row sum $\mathbf{P}^T\mathbbm{1}$ of the optimal plan $\mathbf{P}$, it spreads widely over the support of $\textbf{b}$, with more mass in the region closer to the support of $\textbf{a}$. As $\epsilon$ becomes smaller, more amount of mass is moved to the left half of the support of $\textbf{b}$.

\begin{figure}[!htbp]
\centerline{
\includegraphics[width=80mm]{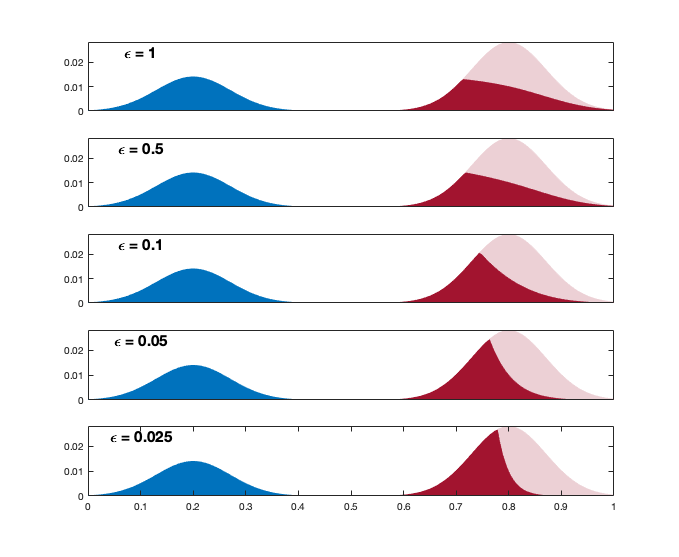}
 }
\caption{The sOT solutions for various $\epsilon$ as $\epsilon \rightarrow 0$. The blue region indicates the transported part of $\textbf{a}$. The dark red region indicates the transported part of $\textbf{b}$, and the light red region corresponds the blocked part of $\textbf{b}$.}
\label{fig:test_eps}
\end{figure}

\subsection{Effect of $C_{\text{cut}}$ for entropy regularized sOT}

In the second example, we test the effect of $\infty$ entries in the cost matrix $\textbf{C}$. In particular, we construct $\textbf{C}$ as in (\ref{eqn:CostMatrix}) but take various cutoff values $C_\text{cut} = 0.25,0.35,0.50,0.55,10$. The marginal densities $\textbf{a}$ and $\textbf{b}$ are 1D Gaussian distributions in (\ref{eqn:Gaussian_equal}).
Note that in this example, we take $\| \textbf{a} \|_1 = \| \textbf{b} \|_1$ such that when there is no $\infty$ entry in $\textbf{C}$, the sOT problem reduces to the standard balanced OT problem. We fix $\epsilon = 0.05$ and $\gamma = 2$ in this example.

\begin{figure}[!htbp]
\centerline{
\includegraphics[width=80mm]{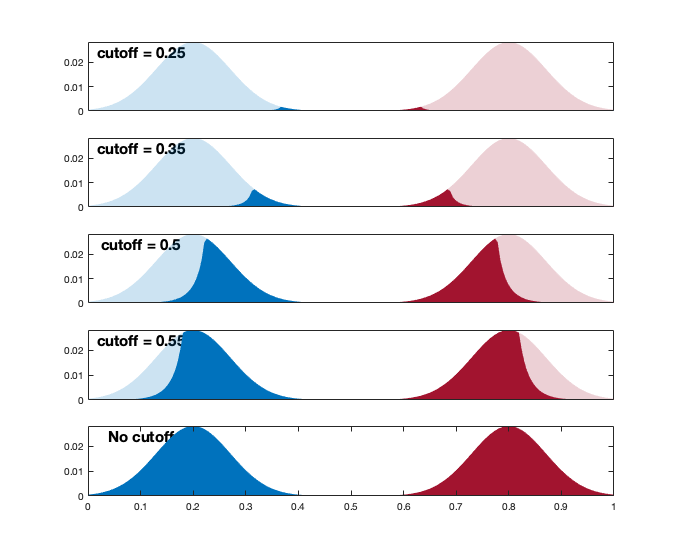}
 }
\caption{The sOT solutions for various cutoff values. The dark blue region indicates the transported part of $\textbf{a}$, and the light blue is for the blocked part of $\textbf{a}$. The dark red region indicates the transported part of $\textbf{b}$, and the light red region corresponds the blocked part of $\textbf{b}$.}
\label{fig:test_cutoff}
\end{figure}
Figure \ref{fig:test_cutoff} presents the solutions of sOT problem with different values of cutoff parameter. When $C_{\text{cut}}$ = 0.25 is small, majority of the mass are blocked, only a tiny amount on the supports of $\textbf{a}$ and $\textbf{b}$ within the separation of 0.25 is allowed to transport. When the $C_{\text{cut}}$ becomes larger (fewer $\infty$ entries in $\textbf{C}$), more amount of mass is transported from the right corner of the support of $\textbf{a}$ to the left corner of the support of $\textbf{b}$. When cutoff becomes large enough, for instance $C_{\text{cut}} = 10$, all entries of $\textbf{C}$ become finite, and the sOT problem degenerates to the standard balanced OT problem, in which all mass in $\textbf{a}$ is completely transported to $\textbf{b}$.

\begin{figure}[!htbp]
\centerline{
\includegraphics[width=55mm]{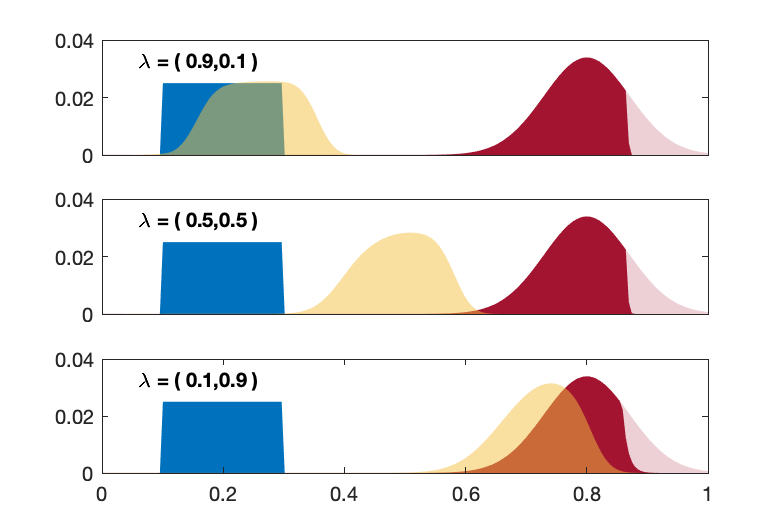}
\includegraphics[width=55mm]{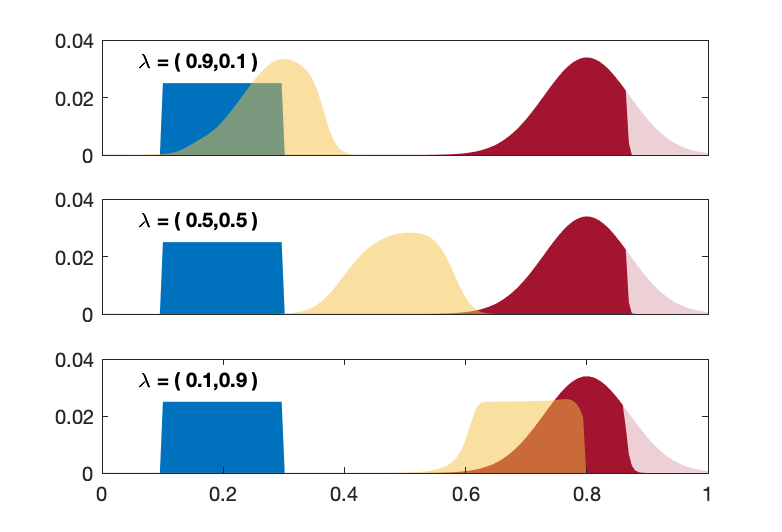}
\includegraphics[width=55mm]{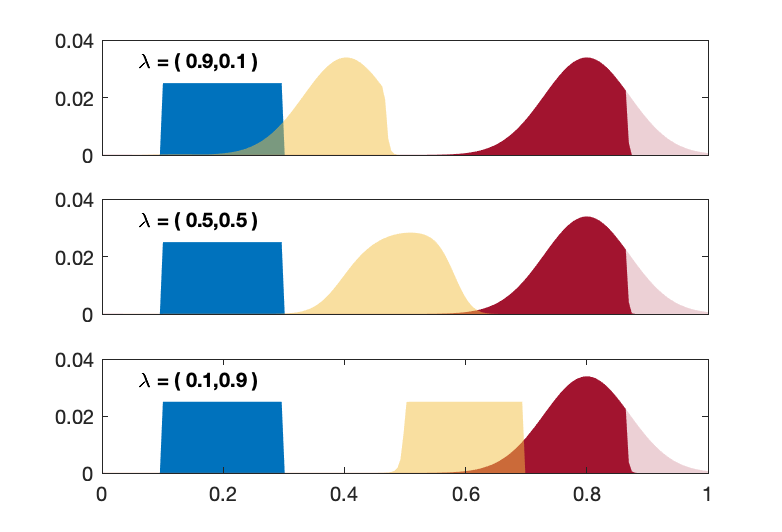} \\
 }
\centerline{
\includegraphics[width=55mm]{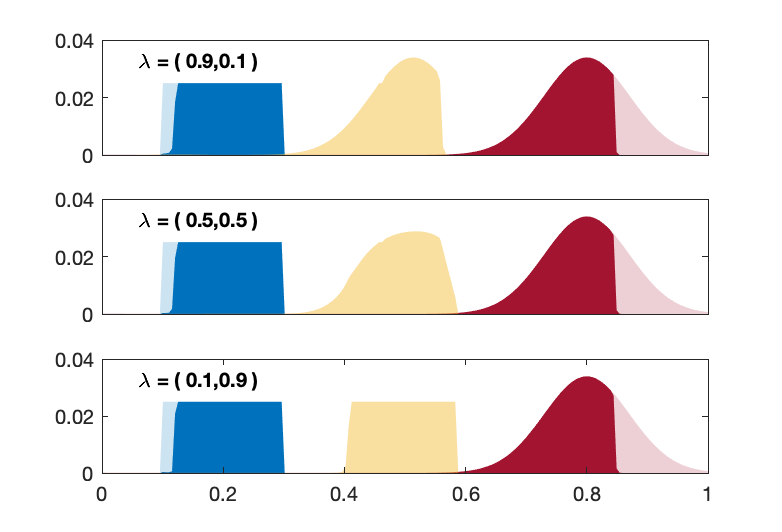}
\includegraphics[width=55mm]{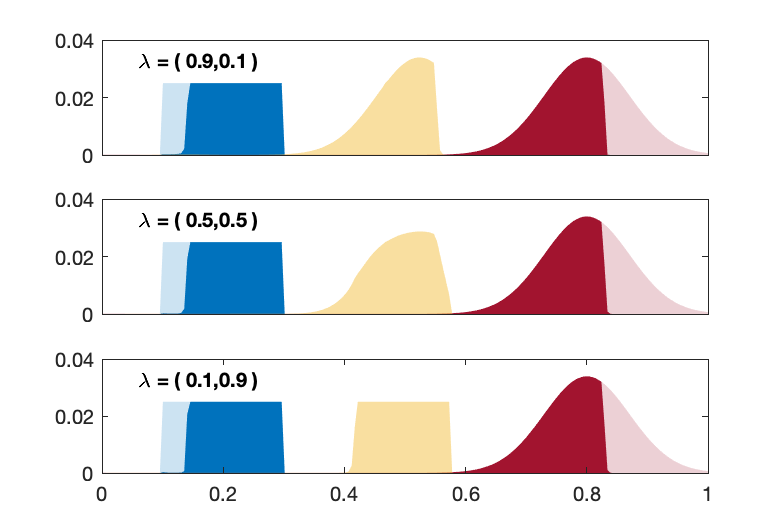}
\includegraphics[width=55mm]{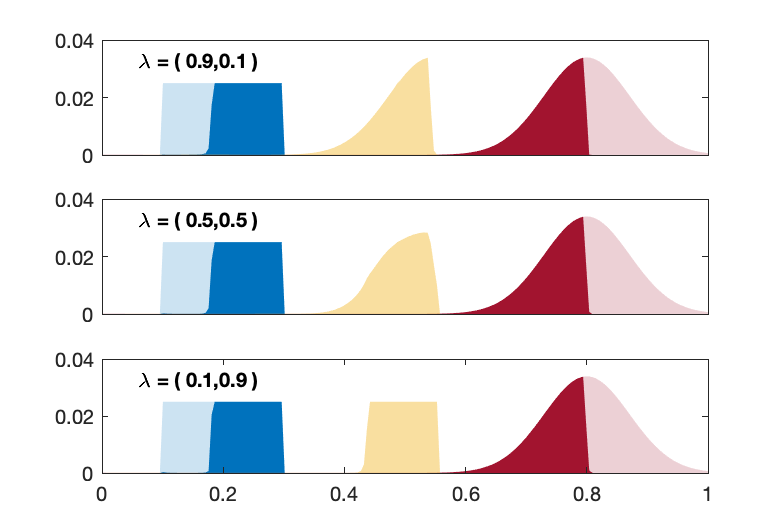} \\
 }
 \centerline{
 \includegraphics[width=55mm]{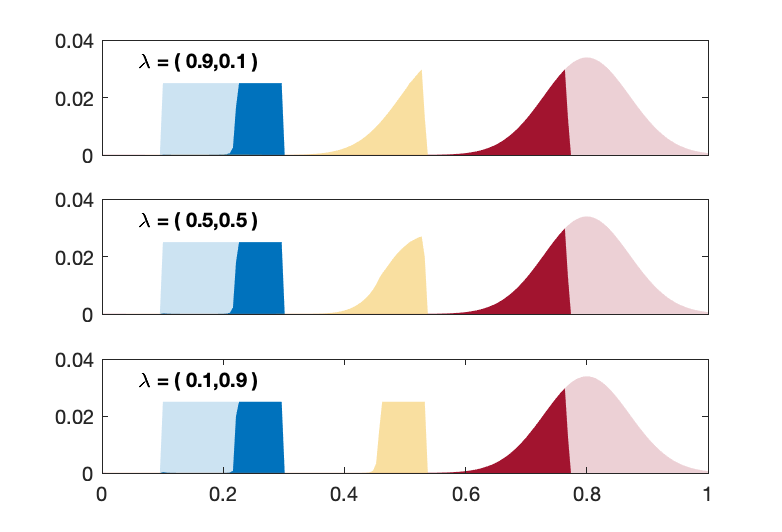}
\includegraphics[width=55mm]{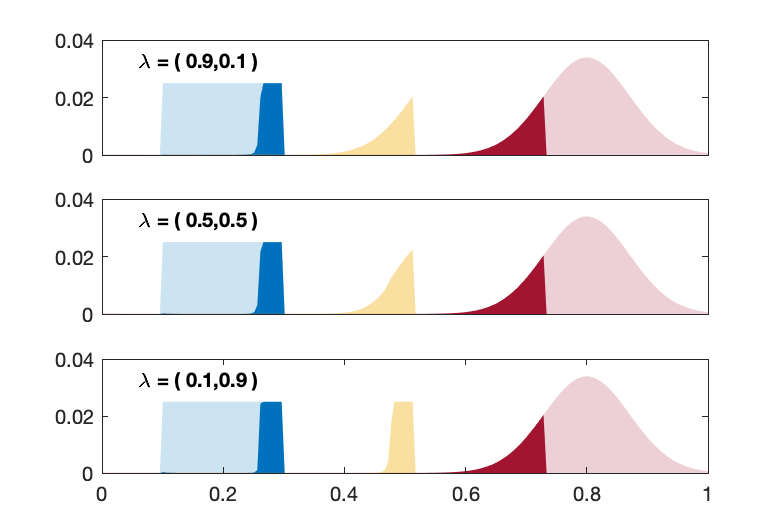}
\includegraphics[width=55mm]{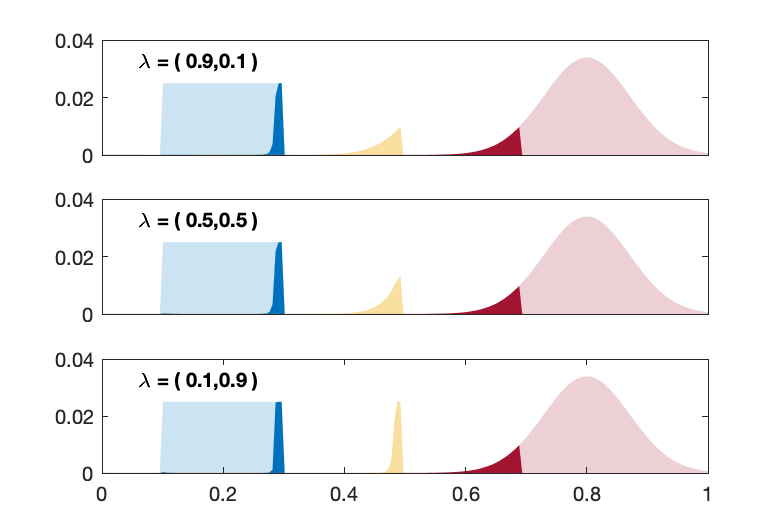} \\
 }
\caption{The weighted sOT barycentric solutions for $(\mathbf{b}_1, \mathbf{b}_2)$ for various values of $C_{\text{cut}}$ and weights $(\lambda_1, \lambda_2)$. From top left, top middle, top right, until bottom right, with $C_\text{cut} =$ 1.00, 0.50, 0.40, 0.30, 0.28, 0.26, 0.24, 0.22, 0.20, respectively, each subfigure consists of three cases  with $(\lambda_1,\lambda_2) = (0.9,0.1), (0.5,0.5), (0.1,0.9)$.}
\label{fig:barycenter_01}
\end{figure}

\begin{figure}[!htbp]
\centerline{
\includegraphics[width=150mm]{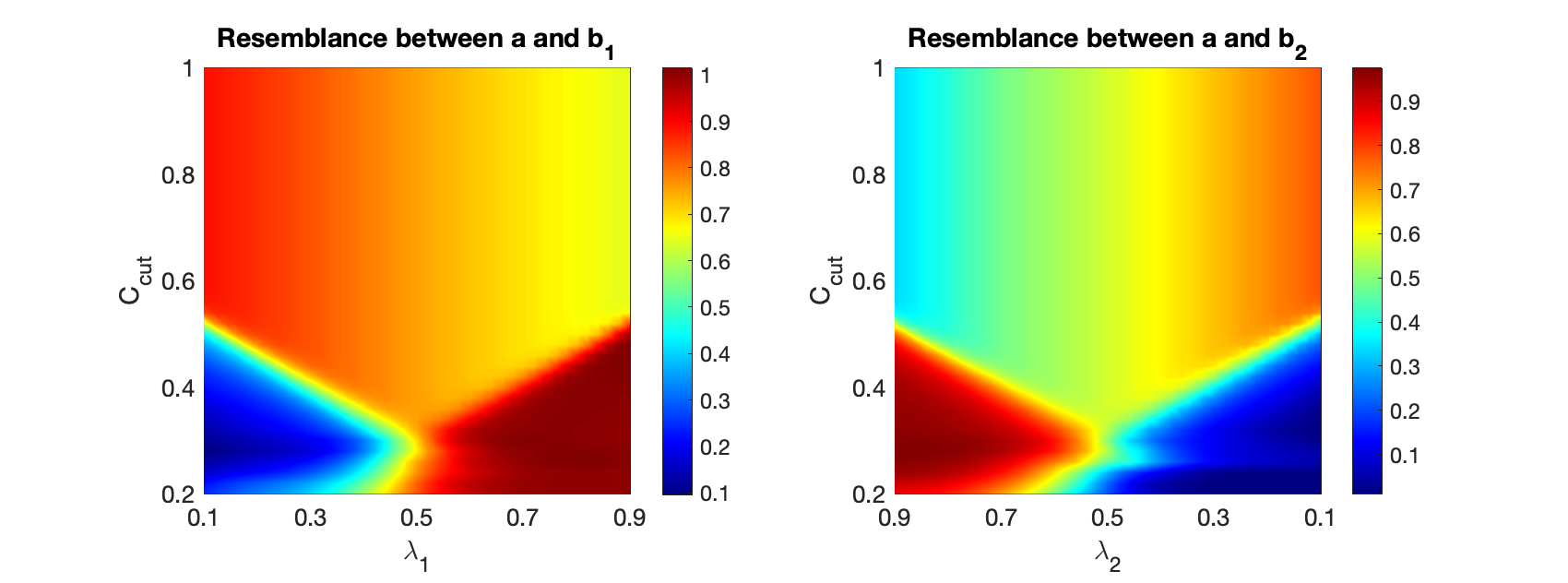}
 }
\caption{The resemblance between the sOT barycenter $\mathbf{a}$ and the marginal distributions $\mathbf{b}_1, \mathbf{b}_2$. When $C_{\text{cut}}$ is small, it shows the reverse and portion selection mechanism as indicated by the limits in Theorem \ref{theorem:mirrority}. When $C_{\text{cut}}$ increases, the mechanism gradually changes back to normal as the standard OT barycenter problem.}
\label{fig:barycenter_02}
\end{figure}

\subsection{The effect of $C_{\emph{cut}}$ on weighted sOT barycenter problem}

In this example, two unequal densities $\mathbf{b}_1,\mathbf{b}_2 \in \mathbb{R}^{200}$ are taken as the discretization of: 
\[
\mathbf{b}_1 = \frac{1}{D_1}\left( \chi_{[0.1,0.3]} +0.001\right)\bigg|_{h\mathbb{Z}\cap[0,1]}, \quad \mathbf{b}_2 = \frac{1}{D_2}\left(e^{-\frac{(x-0.8)^2}{0.1^2}}+0.001\right)\bigg|_{h\mathbb{Z}\cap[0,1]}
\]
over a uniform mesh $\{ x_j = jh \}_{j=0}^{200}$ with $h = \frac{1}{200}$. The constants $D_1, D_2$ are taken such that $\|\mathbf{b}_1\|_1 = 1.0$ and $\|\mathbf{b}_2\|_1 = 1.2$. The cost matrix $\textbf{C} = (C_{ij})\in\mathbb{R}^{200\times 200}$ is taken as in (\ref{eqn:CostMatrix}) with various cutoff values. The value of $\gamma$ is taken to be $\gamma=2$.  For each cutoff value, three pairs of weights $(\lambda_1,\lambda_2) = (0.9,0.1), (0.5, 0.5), (0.1, 0.9)$ are considered. Figure \ref{fig:barycenter_01} depicts the numerical simulations for $C_{\text{cut}} =$ 1.00, 0.50, 0.40, 0.30, 0.28, 0.26, 0.24, 0.22, 0.20, from top left, top middle, top right, until bottom right. In each plot, the opaque light blue (and red, respectively) represents the density $\mathbf{b}_1$ (and $\mathbf{b}_2$, respectively), the transparent dark blue (and red, respectively) represents the transported mass of $\mathbf{b}_1$ (and $\mathbf{b}_2$, respectively), and the transparent yellow is the weighted sOT barycenter.

For a large $C_{\text{cut}} = 1.00$ in the top left subfigure, when $\lambda_1$ is close to 1, the sOT barycenter $\mathbf{a}$ is close to $\mathbf{b}_1$. However $\mathbf{b}_2$ is only partially transported as the mass of $\mathbf{b}_2$ is more than that of $\mathbf{b}_1$. When $\lambda_1$ approaches 0, the sOT barycenter becomes close to partially transported $\mathbf{b}_2$. This is similar to the standard barycentric problem, except that each density $\mathbf{b}_j$ may be only partially transported.

For a small $C_{\text{cut}} = 0.50$ in top middle subfigure, when $\lambda_1$ is close to 1, the sOT barycenter $\mathbf{a}$, compactly supported near the compact domain of $\mathbf{b}_1$, is surprisingly similar to $\mathbf{b}_2$ (up to a translational shift). On the contrary when $\lambda_1$ is close to 0,  the sOT barycenter $\mathbf{a}$, compactly supported near the compact domain of $\mathbf{b}_2$, resembles $\mathbf{b}_1$, up to a translational shift.

When taking further small value $C_{\text{cut}} = 0.24$ as in bottom left subfigure, each density $\mathbf{b}_j$ is only allowed to transport within the distance = 0.24. Hence for any pair $(\lambda_1,\lambda_2)$, the sOT barycenter $\mathbf{a}$ can only be compactly supported in between $\mathbf{b}_1$ and $\mathbf{b}_2$. Besides, without considering the translational shift, $\mathbf{a}$ is valued closely to the transported $\mathbf{b}_2$ when $\lambda_1\rightarrow 1$, while $\mathbf{a}$ is valued closely to $\mathbf{b}_1$ when $\lambda_1\rightarrow 0$. The sOT barycenter solution in bottom right subfigure for $C_{\text{cut}} = 0.20$ is similar to that for $C_{\text{cut}} = 0.24$, except that more mass is blocked for $\mathbf{b}_1$ and $\mathbf{b}_2$.

In Figure \ref{fig:barycenter_02}, we further present the resemblance phase diagram between the sOT barycenter $\mathbf{a}$ and the marginal distribution $\mathbf{b}_1$ on the $\lambda_1$-$C_{\text{cut}}$ plane (the left subfigure), and similarly that between $\mathbf{a}$ and $\mathbf{b}_2$ on the $\lambda_2$-$C_{\text{cut}}$ plane (the right subfigure). Here the resemblance between $\mathbf{a}$ and $\mathbf{b}_j$ is defined as
\[
\text{Resem}(\mathbf{a},\mathbf{b}_j) : = \frac{\min_{0\le k \le 200} \|\mathbf{b}_j - \texttt{circshift}(\mathbf{a},k)\|_2}{\min_{0\le k \le 200} \|\mathbf{b}_1 - \texttt{circshift}(\mathbf{b}_2,k)\|_2}, \quad j=1,2,
\]
in which \texttt{circshift} is a circular shift operator. The smaller the resemblance value is, the more resemblant the two densities are. If $\text{Resem}(\mathbf{x},\mathbf{y}) = 0$, then $\mathbf{x}=\mathbf{y}$ up to a translational shift. It is evident that when $C_{\text{cut}}$ is small, the resemblance of the sOT barycenter $\mathbf{a}$ to $\mathbf{b}_1$ (and $\mathbf{b}_2$, respectively) as $(\lambda_1,\lambda_2)\rightarrow (1,0)$ (and $(\lambda_1,\lambda_2)\rightarrow(0,1)$, respectively) is reversed, compared to the standard OT barycenter problem in which $\mathbf{a}$ resembles $\mathbf{b}_1$ when $\lambda_1 = 1$, and $\mathbf{a}$ resembles $\mathbf{b}_2$ when $\lambda_2 = 1$. On the other hand, as the value of $C_{\text{cut}}$ increases, the reverse effect is lessened. When $C_{\text{cut}}$ becomes sufficiently large such that $\mathbf{C}$ contains no $\infty$ entries, the sOT barycenter problem degenerates to the standard OT barycenter problem, and the reverse mechanism turns back to normal.

\subsection{Color transfer}
Finally, we apply sOT to an important class of image processing problem, the color transfer problem. Specifically, color transfer imposes the color of a target image to an input image so that the output image has the same pattern and geometry as the input image but with the color palette from the target image. This can be viewed as transferring the histogram of pixels in the 3D color space of an image to another \cite{pitie2005n} which optimal transport is powerful at. Direct application of conventional optimal transport causes issues, and several optimal transport based algorithms have been introduced to resolve these drawbacks. For example, adding regularization helps increase the robustness and eliminates outliers \cite{ferradans2014regularized,rabin2014adaptive}. Another issue is that transferring the entire color palette that is very different from the input image results in unrealistic looks. Fixing the amount of transferred mass a priori can help mitigate this issue but with the need of deciding a scale for each application case \cite{bonneel2019spot}. With sOT, we are able to directly control the similarity of transferred color by setting a distance threshold of transferred color in the color space. As a result, we control the color palette similarity of the output and the input image.

For an image, we represent the $n$ pixels as a point cloud $\mathbf{X}\in\mathbb{R}^{n\times 3}$ in the 3-dimensional color space (the RGB space). Given two images represented by $\mathbf{X}\in\mathbb{R}^{n\times 3}$ (input image) and $\mathbf{Y}\in\mathbb{R}^{m\times 3}$ (target image), the color transfer problem is formulated as coupling two uniform distributions $\mathbf{a}\in\mathbb{R}^{n}_+$ and $\mathbf{b}\in\mathbb{R}^m_+$ with the cost matrix $\mathbf{C}\in\mathbb{R}^{n\times m}_+$ where $C_{ij}=\|X_i-Y_j\|^2_2$. In sOT, we use a modified cost matrix $\bar{\mathbf{C}}$ such that $\bar{C}_{ij}=C_{ij}$ for $C_{ij}\leq C_{\text{cut}}$ and $\bar{C}_{ij}=\infty$ otherwise. When dealing with large images, a subsampling and upsampling is often implemented to improve efficiency \cite{ferradans2014regularized}. We first obtain subsampled images $\mathbf{X}^\text{s}$ and $\mathbf{Y}^\text{s}$ using the resize function from the PIL package \cite{umesh2012image} with the ANTIALIAS option. The optimal transport map $\mathbf{P}^{\text{s}}$ between the subsampled images is then determined by sOT algorithm. The output image $\mathbf{X}^{\text{out}}\in\mathbb{R}^{n\times 3}$ with the transferred color palette is constructed such that 

\begin{equation}\label{eqn:ColorTransfer}
\mathbf{X}^{\text{out}}_i = (\mathbf{P}^{\text{s}}\mathbf{Y}^{\text{s}})_{N(i)}/a^{\text{s}}_{N(i)} + \mathbf{X}^{\text{s}}_{N(i)}(1- \sum_j \mathbf{P}^{\text{s}}_{N(i),j}/a^{\text{s}}_{N(i)}) - \mathbf{X}^{\text{s}}_{N(i)} + \mathbf{X}_i,
\end{equation}
where $N(i)$ is the index of the pixel in the subsampled image $\mathbf{X}^{\text{s}}$ that is the closest (in color space) to pixel $i$ in the input image, and $a^{\text{s}}_{N(i)}$ is the source distribution of the subsampled image. The color difference between the input and output images is determined by the color difference due to color transfer in the subsampled input and output images. We take $\gamma = 2$ in this example.

\begin{figure}[!htbp]
\centerline{
\includegraphics[width=0.8\textwidth]{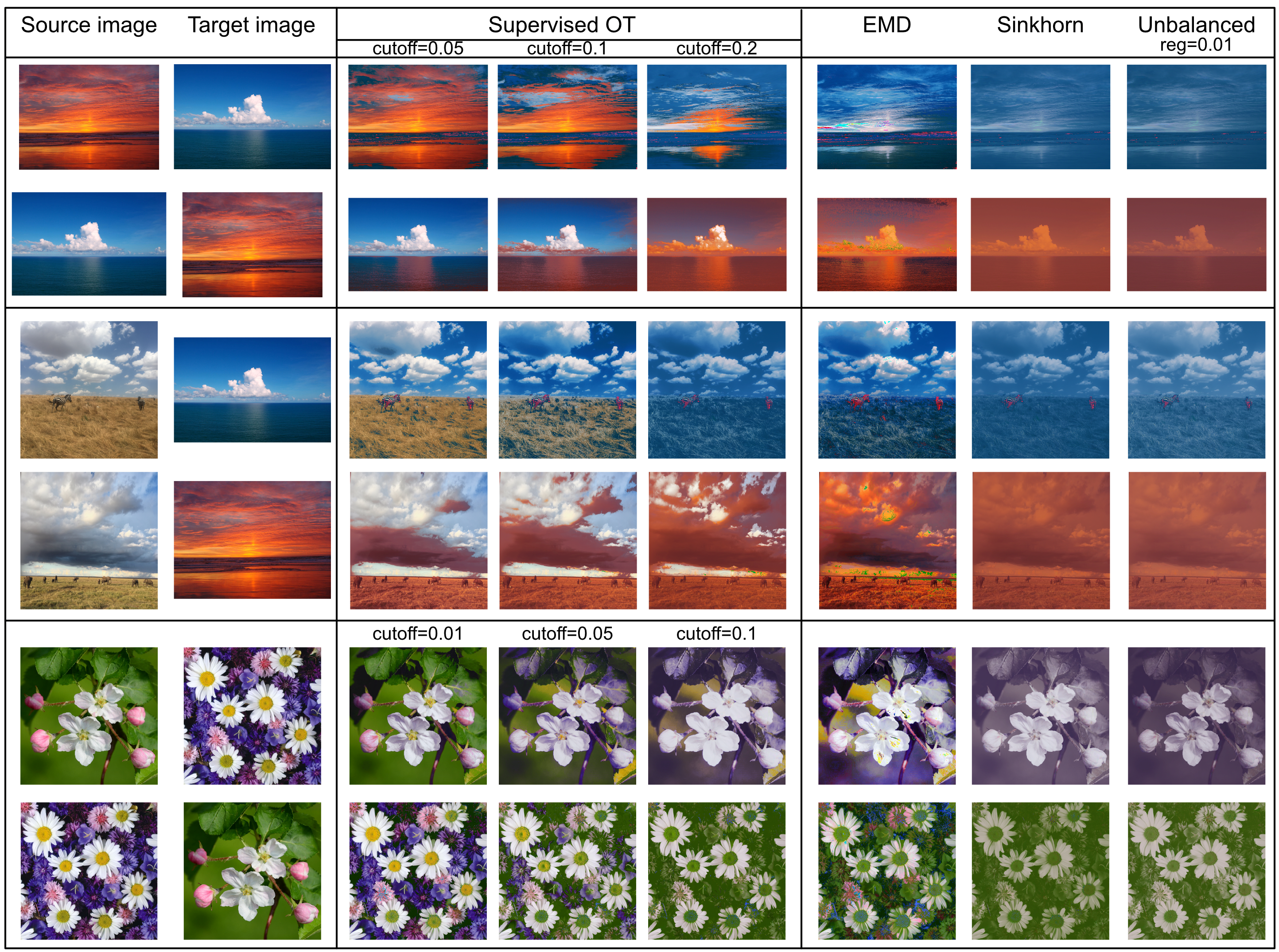}
 }
\caption{The color transfer problem where the color palette of the target image is to be transferred to the source image and the output image keeps the geometry of the source image. The results of supervised optimal transport with different cutoffs in the color space, earth mover's distance, entropy regularized optimal transport (Sinkhorn), and unbalanced optimal transport are shown.}
\label{fig:color_transfer_01}
\end{figure}

The numerical results demonstrate that when the color palettes of the input and target images are considerably distant, applying color transfer produces unrealistic output images (Figure \ref{fig:color_transfer_01}. In contrast, with the supervised cost matrix in sOT, the amount of transferred color can be controlled producing more realistic output images (Figure \ref{fig:color_transfer_01}). When the cutoff value in sOT equals infinity, the output image will converge to the result of regularized optimal transport or earth mover's distance depending on whether entropy regularization was used.

\section{Conclusion and Discussion}

In this work, we introduce the supervised optimal transport problem where the cost matrix $\mathbf{C}$ can have $\infty$ entries and the $\infty$-pattern of $\mathbf{C}$ supervises and controls the transport plans. Also, the source and target distributions need not to be normalized in sOT. These properties make sOT a method generally applicable to a large class of transportation problems where application-specific constraints are to be imposed on the transport plan and the original units are to be preserved for the original distributions. 
To apply sOT on large-scale real problems, we develop a fast numerical solver for sOT based on the Dykstra algorithm. We also extend the OT barycenter problem into a supervised one in the setup of sOT. By considering the sOT barycenter for two distributions $(\textbf{b}_1, \textbf{b}_2)$, a new reverse and portion selection mechanism is discovered, giving the barycenter $\mathbf{a}$ opposite to that of the standard OT barycenter problem when the weight approaches the boundary of the unit simplex. The properties and effects of different parameter values of sOT are illustrated numerically with toy examples. We also demonstrate the reverse behavior of sOT barycenter in an extensive numerical example. In an important problem in imaging science, we compare sOT to several other OT variants to demonstrate its unique utility of supervising the transport plan.

This work can be extended in several ways in the future. For example, a supervised Gromov-Wasserstein OT analogous to sOT can be developed and will enable the integration of multiple subsamples of the same system without known inter-sample correspondence.

In the current work, the formulations (\ref{eqn:UUOT_formII}) and (\ref{eqn:UUOT_formIII}) are obtained from the discrete OT setting. Introducing the entropy regularization terms in (\ref{eqn:KLUUOT_formII}), the optimal plan $\mathbf{P}_{\epsilon}$ has a diagonal rescaling form, such that Dykstra algorithm (Sinkhorn type) can be applied to  improve the numerical efficiency significantly. On the other hand, it is also of practical interest to find an efficient solver for the sOT problem without entropy regularizaion. Inspired by \cite{LiYinOsher_JSC2018}, we may link (\ref{eqn:UUOT_formII}) and (\ref{eqn:UUOT_formIII}) with the dynamical OT (Benamou-Breiner type \cite{BenamouBrenier_NM2018}) formulation. Discrete sOT formulations (\ref{eqn:UUOT_formII}) and (\ref{eqn:UUOT_formIII}) work for any ``ground metric" $\mathbf{C}$. On the other hand,  if the metric $\mathbf{C}$ is homogeneous of degree one such as $L^1$ metric $\mathbf{C} = (\|x_i-y_j\|_1)_{ij}$, the original OT problem can be reformulated as a minimal flux minimization problem \cite{LiRyuOsherYinGangbo_JSC2018}. Motivated by this, it is also interesting to consider, when taking some degree-one homogeneous metric $\mathbf{C}$ in (\ref{eqn:UUOT_formIII}) with certain $\infty$-pattern, whether there exists a minimal flux formulation for sOT problem. Then we can apply well-established efficient solvers for $L^1$ minimization problems to sOT provided that it can be reformulated as a minimal flux problem. Such a potential dynamical formulation for the sOT problem can be used to seek optimal transportation paths with constraints in applications such as continent movement in geology.

\section{Acknowledgements}

Z. Cang’s work is supported by a startup grant from North Carolina State University and NSF grant DMS2151934. Q. Nie is supported by a NSF grant DMS1763272, a grant from Simons Foundation (594598, QN), and a NIH grant U01AR073159. Y. Zhao's work is supported by a grant from the Simons Foundation through Grant No. 357963 and NSF grant DMS-2142500.

\end{document}